\documentclass[12py]{amsart}
\pdfoutput=1
\usepackage[margin=2.5cm]{geometry}
\usepackage{tikz}
\usetikzlibrary{calc}
\usepackage{float}
\usepackage{amsmath}
\usepackage{amsfonts}
\usepackage{amsthm}
\usepackage{enumitem}
\usepackage{dsfont}
\usepackage{amssymb}
\usepackage{pifont}
\usepackage{mathtools}
\usepackage{comment}
\usepackage{graphicx} 
\usepackage{float}
\usetikzlibrary{matrix}

\usepackage{pinlabel} 

\usepackage{hyperref}

\newtheorem{thm}{Theorem}[section]
\newtheorem{lem}[thm]{Lemma}
\newtheorem{prop}[thm]{Proposition}

\newtheorem{cor}[thm]{Corollary}
\newtheorem{exam}[thm]{Example}
\theoremstyle{definition}

\theoremstyle{definition}

\newtheorem{defn}[thm]{Definition}

\newtheorem{remark}[thm]{Remark}

\newtheorem{note}[thm]{Note}

\newtheorem*{claim*}{Claim}
\newtheorem*{quest*}{Question}
\newtheorem*{remark*}{Remark}
\newtheorem*{fact*}{Fact}

\newcommand{\Z}{\ensuremath{\mathbb{Z}}}
\newcommand{\Q}{\ensuremath{\mathbb{Q}}}
\newcommand{\R}{\ensuremath{\mathbb{R}}}

\title{On the symmetric group action on rigid disks on a strip}
\author{Nicholas Wawrykow}
\thanks{Department of Mathematics, University of Michigan, 530 Church St, Ann Arbor, MI 48109   \\\href{mailto:wawrykow@umich.edu}{wawrykow@umich.edu} Support: The author was supported in part by NSF grant DMS - 1840234}
\date{}
\begin{document}

\maketitle
\begin{abstract}
In this paper we decompose the rational homology of the ordered configuration space of $n$ open unit-diameter disks on the infinite strip of width $2$ as a direct sum of induced $S_{n}$-representations. In \cite{alpert2020generalized}, Alpert proves that the $k^{\text{th}}$-integral homology of the ordered configuration space of $n$ open unit-diameter disks on the infinite strip of width $2$ is an FI$_{k+1}$-module by studying certain operations on homology called ``high-insertion maps." The integral homology groups $H_{k}\big(\text{Conf}(n,2)\big)$ are free abelian, and Alpert computes a basis for $H_{k}\big(\text{Conf}(n,2)\big)$ as an abelian group. In this paper, we study the rational homology groups as $S_{n}$-representations. We find a new basis for $H_{1}\big(\text{Conf}(n,2);\Q\big),$ and use this, along with results of Ramos \cite{ramos2017generalized}, to give an explicit description of $H_{k}\big(\text{Conf}(n,2);\Q\big)$ as a direct sum of induced $S_{n}$-representations arising from free FI$_{d}$-modules. We use this decomposition to calculate the rank of the rational homology of the unordered configuration space of $n$ open unit-diameter disks on the infinite strip of width $2$.
\end{abstract}

\section{Introduction}
The ordered configuration space of $n$ open unit-diameter disks on the infinite strip of width $w$, denoted $\text{Conf}(n,w)$, is the space of ways of placing $n$, labeled, non-overlapping, unit-diameter disks on an infinite strip of width $w$. One can parameterize $\text{Conf}(n,w)$ as subspace of $\R^{2n}$
\[
\text{Conf}(n,w):=\Big\{(x_{1}, y_{1}, \dots, x_{n}, y_{n})\in \R^{2n}|(x_{i}-x_{j})^{2}+(y_{i}-y_{j})^{2}\ge 1\text{ for }i\neq j, \text{ and }\frac{1}{2}\le y_{i}\le w-\frac{1}{2} \text{ for all }i\Big\},
\]
where the point $(x_{1}, y_{1},\dots, x_{n}, y_{n})$ corresponds the configuration of disks where disk $i$ is centered at $x_{i}, y_{i}$, see Figure \ref{PointinConf(7,2)} for an example of a point in $\text{Conf}(7,2)$. 

\begin{figure}[h]
\centering
\includegraphics[width = 10cm]{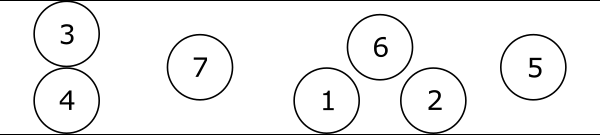}
\caption{A point in $\text{Conf}(7,2)$.}
\label{PointinConf(7,2)}
\end{figure}

Baryshnikov, Bubenik, and Kahle's paper \cite{BBK} began the study of the topology of disk configuration spaces, though interest in these spaces outside of pure mathematics predates this. Disk configurations spaces have been studied in the context of robotics and motion planning, phase transitions in physical chemistry, and physics, for example see \cite{farber2008invitation}, \cite{diaconis2009markov}, and \cite{carlsson2012computational}.

When there are at most as many disks as the strip is wide, i.e., $n\le w$, there are no obstructions to moving disks by each other, and $\text{Conf}(n,w)$ is homotopic to $F_{n}(\R^{2})$, the ordered configuration space of $n$ points in the plane. The (co)homology of these spaces are well studied, having explicit descriptions given in \cite{arnold1969cohomology} and \cite{cohen2007homology}; see, for example, \cite{knudsen2018configuration} or \cite{sinha2006homology} for an overview. When the number of disks is greater than the width of the strip, i.e., $n>w$, the complexity of the topology of $\text{Conf}(n,w)$ increases dramatically, as barriers to the free movement of the disks arise.

In their study of the homology of $\text{Conf}(n,w)$, Alpert, Kahle, and MacPherson proved that for all $n$ and $w$ the space $\text{Conf}(n,w)$ is homotopy equivalent to a polyhedral cell complex called  $\text{cell}(n,w)$ \cite[Theorem 3.1]{alpert2021configuration}. By studying this cell complex, they were able to determine the order of growth of $H_{k}\big(\text{Conf}(n,w)\big)$ as a function of $n$ for fixed $k$, proving it to be exponential times polynomial. When $n>w$, the exponential term in the growth rate is nontrivial, and the growth of $H_{k}\big(\text{Conf}(n,w)\big)$ outpaces that of $H_{k}\big(F_{n}(\R^{2})\big)$, which is only polynomial in $n$. 

As in the case of the ordered configuration space of points on the plane, the symmetric group $S_{n}$ acts on $\text{Conf}(n,w)$ by relabeling the disks. This $S_{n}$-action corresponds to an $S_{n}$-action on $\text{cell}(n,w)$. The symmetric group actions on $F_{n}(\R^{2})$ and $\text{cell}(n,w)$ extend to actions on $H_{k}\big(F_{n}(\R^{2})\big)$ and $H_{k}\big(\text{cell}(n,w)\big)$ for all $n$, making these homology groups symmetric group representations. There are maps from $H_{k}\big(F_{n}(\R^{2})\big)$ to $H_{k}\big(F_{n+1}(\R^{2})\big)$ that make $H_{k}\big(F_{\bullet}(\R^{2})\big)$ a finitely generated FI-module \cite[Theorem 6.4.3]{church2015fi}. This cannot be the case for $H_{k}\big(\text{cell}(\bullet,w)\big)$ as the rank of a finitely generated FI-module grows at most polynomially in $n$. Alpert proved that $H_{k}\big(\text{cell}(\bullet,2)\big)$ and therefore $H_{k}\big(\text{Conf}(\bullet,w)\big)$, is a finitely generated module over a category called FI$_{k+1}$, a generalization of FI \cite[Theorem 6.1]{alpert2020generalized}.

\begin{thm}
(Alpert \cite[Theorem 6.1]{alpert2020generalized}) For any $k$, the homology groups $H_{k}\big(\text{Conf}(\bullet,2)\big)$ form a finitely generated $FI_{k+1}$-module over $\Z$.
\end{thm}

This theorem, along with results of Ramos \cite{ramos2017generalized}, suggests that it might be possible to determine a decomposition of $H_{k}\big(\text{Conf}(\bullet,2)\big)$ into irreducible symmetric group representations. Unfortunately, several computational difficulties arise when we consider homology with integer coefficients. Since $H_{k}\big(\text{Conf}(\bullet,2)\big)$ is torsion free \cite[Theorem 4.3]{alpert2020generalized} and symmetric group representations are better behaved over $\Q$ than $\Z$, we study $H_{k}\big(\text{Conf}(\bullet,2);\Q\big)$ in order to avoid some of these difficulties. We decompose $H_{k}\big(\text{Conf}(n, 2);\Q\big)$ as a direct sum of induced symmetric group representations.

\begin{thm}\label{mainthm}
As a representation of the symmetric group $S_{n}$, 
\[
H_{k}\big(\text{Conf}(n,2);\Q\big)\cong \bigoplus_{\substack{(n_{1},\dots, n_{k})|\\n_{i}\ge 2, \sum_{j=1}^{k}n_{j}\le n}}\Big(\bigoplus_{\substack{a=(a_{1},\dots, a_{d+1})|\\\sum a_{j}=n-\sum n_{i}\\d=\#n_{i}|n_{i}=2}}\big(\text{Ind}^{S_{n}}_{S_{n_{1}}\times\cdots\times S_{n_{k}}\times S_{a}}W_{n_{1}}\boxtimes\cdots\boxtimes W_{n_{k}}\boxtimes \Q\big)\Big),
\]
where $W_{n_{j}}=Triv_{2}$ is the trivial representation of $S_{2}$ if $n_{j}=2$, $W_{n_{j}}=\bigwedge^{2}Std_{n_{j}}$ where $Std_{n_{j}}$ is the standard representation of $S_{n_{j}}$  if $n_{j}>2$, and $\Q$ denotes the trivial representation of $S_{a}:=S_{a_{1}}\times\cdots \times S_{a_{d+1}}$.
\end{thm}

Since we can use the Littlewood-Richardson rule to decompose $\text{Ind}^{S_{n}}_{S_{n_{1}}\times\cdots\times S_{n_{k}}\times S_{a}}W_{n_{1}}\boxtimes\cdots\boxtimes W_{n_{k}}\boxtimes \Q$ into irreducible $S_{n}$-representations, it follows that one could use Theorem \ref{mainthm} to get a decomposition of $H_{k}\big(\text{Conf}(n,2);\Q\big)$ into irreducible $S_{n}$-representations. From this direct sum we deduce formulae for the multiplicity of the trivial and alternating representations of $S_{n}$ in $H_{k}\big(\text{Conf}(n,2);\Q\big)$. We use the former to determine the rational homology of the unordered configuration space of $n$ disks on an infinite strip of width $2$.

\begin{cor}\label{UnorderCor}
Let $\text{UConf}(n,w):=\text{Conf}(n,w)/S_{n}$ denote the unordered configuration space of $n$ unit-diameter disks on the infinite strip of width $w$. Then $H_{k}\big(\text{UConf}(n,2);\Q\big)=\Q^{\binom{n-k}{k}}$.
\end{cor}

To prove Theorem \ref{mainthm} we provide a basis for $H_{1}\big(\text{cell}(n,2);\Q\big)$ that allows us to easily decompose it into irreducible $S_{n}$-representations. This basis is not the basis of Alpert's Theorem 4.3 when viewed as a basis for rational homology \cite[Theorem 4.3]{alpert2020generalized}. We then show that certain products of basis elements yield generating sets for free FI$_{d}$-modules. We then show that the various free FI$_{d}$-modules that we get are linearly independent and generate $H_{k}\big(\text{cell}(n,2);\Q\big)$. This, along with a result of Ramos \cite[Proposition 3.4]{ramos2017generalized}, yields Theorem \ref{mainthm}.

\subsection{Acknowledgements}
The author would like to thank Jennifer Wilson for her insights and encouragement. The author would also like to thank Andrew Snowden for helpful conversations. The author was supported in part by NSF grant DMS - 1840234.

\section{$\text{Conf}(n,w)$ and $\text{cell}(n,w)$}

We review a theorem of Alpert, Kahle, and MacPherson demonstrating the homotopy equivalence between the ordered configuration space of $n$, open, unit-diameter disks on the infinite strip of width $w$ and the polyhedral cell complex $\text{cell}(n,w)$ \cite[Theorem 3.1]{alpert2021configuration}. After that we recall Alpert's theorem describing a basis for $H_{k}\big(\text{cell}(n,2)\big)$ \cite[Theorem 4.3]{alpert2020generalized}. This theorem will provide us hints on how to decompose elements of $H_{k}\big(\text{cell}(n,2);\Q\big)$ into products of elements in $H_{1}\big(\text{cell}(\bullet,2);\Q\big)$ making our study of the $S_{n}$-representation structure of $H_{k}\big(\text{Conf}(n,2);\Q\big)$ simpler.

Recall the definition of $\text{Conf}(n,w)$ from the introduction. Directly
computing the homology groups of these configuration spaces is a hard task. Fortunately, $\text{Conf}(n,w)$ is closely related to a polyhedral cell complex called $\text{cell}(n,w)$. We recall the definition of this complex that will help us investigate the structure of the homology groups of $\text{Conf}(n,2)$. Before doing that, we first define \emph{symbols}, which will serve as labels for the cells of $\text{cell}(n)$ and $\text{cell}(n,w)$

\begin{defn}
A \emph{symbol} on the set $\{1,\dots, n\}$ is a string of numbers and vertical bars, such that the numbers form a permutation of $1$ through $n$, and such that each vertical bar is immediately preceded and followed by a number. The elements $1,\dots, n$ are call \emph{labels}.
\end{defn}

\begin{defn}
A \emph{block} in a symbol is a substring of numbers between two vertical bars.
\end{defn}

\begin{exam}
$1|3\,2|4$ and $3|4\,1\,2$ are symbols on $\{1,2,3,4\}$, and $1$, $3\,2$, and $4$ are the blocks in $1|3\,2|4$, while $3$ and $4\,1\,2$ are the blocks in $3|4\,1\,2$.
\end{exam}

The cells of $\text{cell}(n)$ are represented by the symbols on $[n]:=\{1, \dots, n\}$, and the cells of $\text{cell}(n,w)$ are represented by the symbols on $[n]$ with blocks of size at most $w$. With this in mind we need a method of describing face and coface relations of cells in $\text{cell}(n)$ and $\text{cell}(n,w)$.

\begin{defn}
Given two blocks of disjoint elements $a_{1}\,a_{2}\,\cdots a_{n}$ and $b_{1}\,b_{2}\,\cdots\, b_{m}$ we can generate a third block on $n+m$ elements by placing the block $b_{1}\,b_{2}\,\cdots\, b_{m}$ after the block $a_{1}\,a_{2}\,\cdots a_{n}$. A \emph{shuffle} of the resulting block $a_{1}\,a_{2}\,\cdots a_{n}\,b_{1}\,b_{2}\,\cdots\, b_{m}$ with respect to $a_{1}\,a_{2}\,\cdots a_{n}$ and $b_{1}\,b_{2}\,\cdots\, b_{m}$ is a permutation of the elements of this block such that if $i<j$, the image of $a_{i}$ is to the left of the image of $a_{j}$, and the image of $b_{i}$ is to the left of the image of $b_{j}$.
\end{defn}

\begin{exam}
Consider the two blocks $1\,3\,2$ and $4\,5$, then $1\,3\,2\,4\,5$ and $4\,1\,3\,5\,2$ are shuffles of $1\,3\,2\,4\,5$ with respect to $1\,3\,2$ and $4\,5$, but $1\,2\,3\,4\,5$ is not a shuffle of $1\,3\,2\,4\,5$ with respect to $1\,3\,2$ and $4\,5$, as the order of the labels of the block $1\,3\,2$ is not preserved.
\end{exam}

Now we define $\text{cell}(n)$ and $\text{cell}(n,w)$.

\begin{defn}
The polyhedral cell complex \emph{$\text{cell}(n)$} has one $m$-cell for every symbol on $[n]$ with $n-m$ blocks. If $f$ and $g$ are two cells in $\text{cell}(n)$, then $f$ is a \emph{face} of $g$ if and only if the symbol representing $g$ can be obtained from the symbol representing $f$ by removing a sequence of bars from $g$ and shuffling the resulting blocks. We say that $g$ is a \emph{coface} of $f$.
\end{defn}

\begin{exam}
The cell complex $\text{cell}(3)$ has $0$-cells $1|2|3$, $1|3|2$, $2|1|3$, $2|3|1$, $3|1|2$, and $3|2|1$; 1-cells $1|2\,3$, $1|3\,2$, $1\,2|3$, $1\,3|2$, $2|1\,3$, $2|3\,1$, $2\,1|3$, $2\,3|1$, $3|1\,2$, $3|2\,1$, $3\,1|2$, and $3\,2|1$; and $2$-cells $1\,2\,3$, $1\,3\,2$, $2\,1\,3$, $2\,3\,1$, $3\,1\,2$, and $3\,2\,1$. See Figure \ref{cell(3)pic} to see the face and coface relations in $\text{cell}(3)$.
\end{exam}

\begin{figure}[h]
\centering
\includegraphics[width = 10cm]{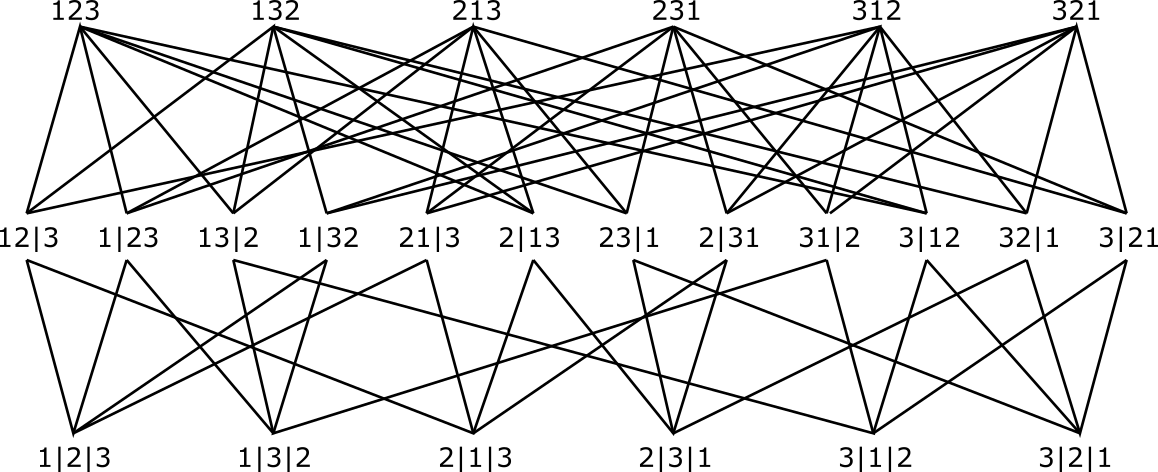}
\caption{The face-coface relation of $\text{cell}(3)$, where the $2$-cells are at the top of the image, $1$-cells in the middle, and $0$-cells on the bottom. Lines connect cofaces to faces.}
\label{cell(3)pic}
\end{figure}

\begin{defn}
The polyhedral cell complex \emph{$\text{cell}(n,w)$} is the subcomplex of $\text{cell}(n)$ that contains every cell $f$ in $\text{cell}(n)$ represented by a symbol whose blocks have at most $w$ elements.
\end{defn}

\begin{exam}
The cell complex $\text{cell}(3,2)$ has $0$-cells $1|2|3$, $1|3|2$, $2|1|3$, $2|3|1$, $3|1|2$, and $3|2|1$; 1-cells $1|2\,3$, $1|3\,2$, $1\,2|3$, $1\,3|2$, $2|1\,3$, $2|3\,1$, $2\,1|3$, $2\,3|1$, $3|1\,2$, $3|2\,1$, $3\,1|2$; and no $2$-cells. See Figure \ref{cell(3,2)pic} to see the cell structure.
\end{exam}

\begin{figure}[h]
\centering
\includegraphics[width = 6cm]{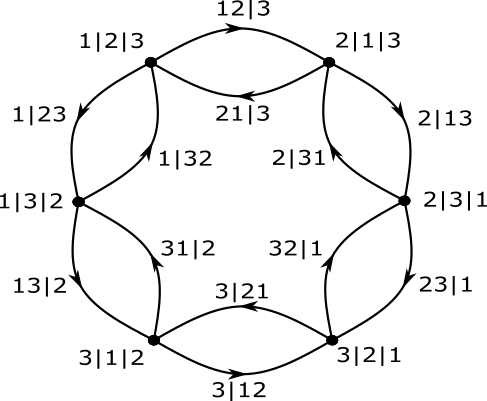}
\caption{The cell structure of $\text{cell}(3,2)$. Note the ``small loops," e.g., $2|1\,3+\,2|3\,1$, are concatenation products of injected cycles.}
\label{cell(3,2)pic}
\end{figure}

Note that some cells in $\text{cell}(n)$ and $\text{cell}(n,w)$ are the product of cells, e.g., $1|2\,3$ is the product of a $0$-cell labeled by $1$ and a $1$-cell labeled by $2\,3$.


\begin{defn}
We can form the \emph{concatenation product} of two cells with disjoint labels by placing them next to each other with a bar in between.
\end{defn}

\begin{exam}
The $2$-cell $32|14$ is the concatenation product of the $1$-cell labeled by $32$ and the $1$-cell labeled by $14$. 
\end{exam}

We can view $\text{cell}(n)$ and $\text{cell}(n,w)$ as chain complexes if we can find a coherent way of determining the signs of the face maps, which would yield boundary maps $\partial$. This can be done inductively, by first determining the signs of the faces of a cell $g$ whose representative symbol  consists of a single block, and then establishing some ``Leibniz rule." 

Let $g$ be a cell represented by a symbol with only $1$ block, and let $f$ be a top dimensional face of $g$, i.e., there is one bar in the symbol representing $f$ and $g$ can be obtained by shuffling the two blocks of $f$. We set the sign of $f$ in $\partial g$ to be the sign of the permutation that takes the symbol representing $g$ to the symbol that arises from deleting the bar in the symbol representing $f$ and not reshuffling. Thus, 
\[
\partial g:=\sum a_{i}f_{i},
\]
where $f_{i}$ is a top dimensional face of $g$, and $a_{i}$ is the sign of the permutation taking the labels of $f_{i}$ to the labels of $g$. We can define the boundary of the concatenation product of cells whose symbols have disjoint sets of elements. Let $b(g)$ denote the number of blocks in $g$, then we set
\[
\partial(g_{1}|g_{2}):=\partial g_{1}|g_{2}+(-1)^{b(g_{1})}g_{1}|\partial g_{2}.
\]
If $g$ is a block size $1$ then set $\partial g=0$. By Lemma 2.1 of Alpert $\partial^{2}=0$, so $\partial$ is a boundary map \cite[Lemma 2.1]{alpert2020generalized}. We use this to calculate the homology groups of $\text{cell}(n)$ and $\text{cell}(n,w)$.

\begin{exam}
Note that $\text{cell}(2,2)$ is a loop, so $H_{0}(\text{cell}(2,2))=\Z$ and $H_{1}(\text{cell}(2,2))=\Z$. We set $z(12)=1\,2+2\,1$ to be the cycle representing a generator of $H_{1}(\text{cell}(2,2))$. See Figure \ref{b2}.
\end{exam}

If $A$ and $B$ are disjoint sets then we can concatenate a class $[z_{A}]\in H_{k_{A}}\big(\text{cell}(A, 2)\big)$ and $[z_{B}]\in H_{k_{B}}\big(\text{cell}(B, 2)\big)$ to get a class in $[z_{A}|z_{B}]\in H_{k_{A}+k_{B}}\big(\text{cell}(A\sqcup B, 2)\big)$ by concatenating the cells of a class representing $[z_{A}]$ with the cells of a class representing $[z_{B}]$. This motivates our next definition.

\begin{defn}
An \emph{injected cycle} is the result of applying an inclusion $\iota:[n]\hookrightarrow[m]$ to the labels of every cell of a cycle representing a class in $H_{i}\big(\text{cell}(n,w)\big)$.
\end{defn}

The majority of the cycles we will use to represent homology classes can be viewed as concatenation products of injected cycles.

\begin{exam}
A $1$-cycle generating $H_{1}\big(\text{cell}(2,2);\Z\big)$ is $1\,2+2\,1$. If $\iota:[2]\hookrightarrow[3]$ is the map that sends $1$ to $3$, and $2$ to $1$, then $3\,1+1\,3=1\,3+3\,1$ is an injected cycle representing a class in $H_{1}\big(\text{cell}(\{1, 3\},2);\Z\big)$. By taking the concatenation product with $2$ (an injected $0$-cycle representing a class in $H_{0}\big(\text{cell}(\{2\},2);\Z\big)$) on either the left (or right), we obtain a $1$-cycle in $H_{1}\big(\text{cell}(3,2)\big)$, namely $2|(1\,3+3\,1)=2|1\,3+2|3\,1$ (or $(1\,3+3\,1)|2=1\,3|2+3\,1|2$). See the small loops in Figure \ref{cell(3,2)pic}.
\end{exam}

Now that we have defined $\text{cell}(n,w)$, we state Alpert, Kahle, and MacPherson's theorem proving that $\text{Conf}(n,w)$ and $\text{cell}(n,w)$ are homotopy equivalent.

\begin{thm}\label{confiscell}
(Alpert--Kahle--MacPherson \cite[Theorem 3.1]{alpert2021configuration}) For every $n, w\ge 1$, we have a homotopy equivalence $\text{Conf}(n,w)\simeq \text{cell}(n,w)$. Moreover, these homotopy equivalences for $w$ and $w+1$ commute up to homotopy with the inclusions $\text{cell}(n,w)\hookrightarrow \text{cell}(n,w+1)$ and $\text{Conf}(n,w)\hookrightarrow \text{Conf}(n,w+1)$.
\end{thm}

This theorem allows us to study $H_{k}\big(\text{cell}(n,w)\big)$ in place of $H_{k}\big(\text{Conf}(n,w)\big)$. We restrict ourselves to $w=2$. 

\begin{defn}
Let $f=\alpha_{1}|\alpha_{2}|\cdots|\alpha_{r}$ be a symbol in $\text{cell}(n)$, where the $\alpha_{i}$ are blocks. We say that a block $\alpha_{i}$ is a \emph{singleton} if it only has one element, and we say that a block $\alpha_{j}$ is a \emph{follower} if the preceding block $\alpha_{j-1}$ is a singleton, whose element is less than every element of $\alpha_{j}$.
\end{defn}

Using discrete Morse Theory, for any $k\ge 0$ one can define a unique cycle $z(e)\in H_{k}\big(\text{cell}(n,2)\big)$ for each symbol $e$ on $n$ labels such that there are $k$ blocks with $2$ elements, every pair of consecutive singletons in $e$ are in decreasing order, and if a $2$-element block in $e$ has its elements in decreasing order, it is a follower. We start by adopting Alpert's conventions for small cycles. There is one type of trivial cycle:  $z(1):=1$. Additionally, there are two fundamental $1$-cycles: $z(1\,2):=1\,2+2\,1$, this is a generating element of $H_{1}\big(\text{cell}(2,2)\big)$, see Figure \ref{b2}, and $z(1|3\,2):= 1|3\,2+3\,1|2+3|2\,1-3\,2|1-2|3\,1-2\,1|3$, i.e., the boundary in $\text{cell}(3)$ of the cell represented by $3\,2\,1$, i.e., the large, inner, blue loop oriented clockwise in Figure \ref{1cycleincell(3,2)}. If $e$ is concatenation product of cells of the form $i_{1}$, $i_{2}\,i_{3}$ where $i_{2}<i_{3}$, and $i_{4}|i_{5}\, i_{6}$ where $i_{5}>i_{6}>i_{4}$, such that $e$ satisfies the ordering restrictions at the beginning of the paragraph we set $z(e)$ to be equal to the concatenation product of the injected cycles $z(i_{1})$, $z(i_{2}\,i_{3})$, and $z(i_{4}|i_{5}\,i_{6})$, e.g., $z(3\,4|7|1|6\,2|5)=z(3\,4)|z(7)|z(1|6\,2)|z(5)$. Alpert proved that cycles of this form are a basis for $H_{k}\big(\text{cell}(n,2)\big)$ \cite[Theorem 4.3]{alpert2020generalized}.

\begin{figure}[h]
\centering
\includegraphics[width = 6 cm]{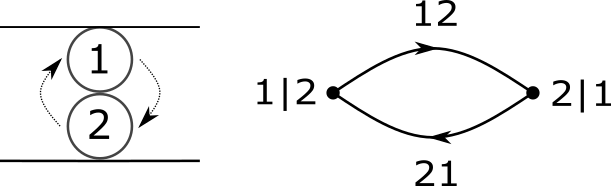}
\caption{On the right, the cycle $z(12)$ in $H_{1}\big(\text{cell}(2,2)\big)$. On the left, the corresponding loop in $\text{Conf}(2,2)$, where the disks labeled $1$ and $2$ perform a full loop around each other.}
\label{b2}
\end{figure}

\begin{figure}[h]
\centering
\includegraphics[width = 6 cm]{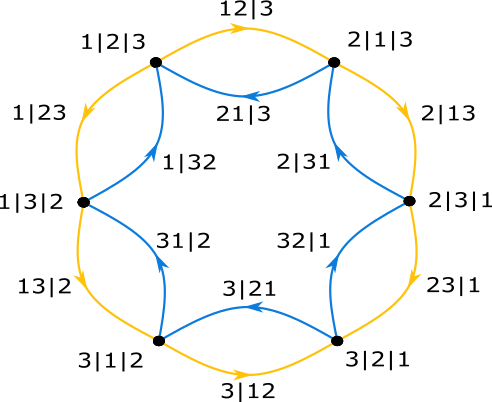}
\caption{The cycle $z(1|3\,2)$ is the blue inside loop oriented clockwise, the cycle $z'(1|3\,2)$ is the average over the blue and yellow loops, both of which are oriented clockwise,}
\label{1cycleincell(3,2)}
\end{figure}

\begin{thm}\label{basistheorem}
(Alpert \cite[Theorem 4.3]{alpert2020generalized}) For $k\ge 0$, a basis for $H_{k}\big(\text{cell}(n,2)\big)$ is given by the classes of the cycles $z(e)$ where $e$ is a symbol on $[n]$ with the following properties:
	\begin{enumerate}
	\item The largest blocks in $e$ have at most $2$ elements.
	\item There are $k$ blocks with $2$ elements in $e$.
	\item Every two consecutive singletons in $e$ are in decreasing order.
	\item If a given 2-element block has its elements in decreasing order, then the block is a follower in $e$.
	\end{enumerate}
\end{thm}

\begin{defn}
We say that a cell $e\in \text{cell}(n,2)$ is \emph{critical} if it satisfies the requirements of Theorem \ref{basistheorem}.
\end{defn}

The symmetric group acts by permuting the labels of a cell in $\text{cell}(n,w)$ and this action extends to $H_{k}\big(\text{cell}(n,w)\big)$. Unfortunately, it is rather hard to determine the exact structure of $H_{k}(\text{cell}(n,2))$ as an $S_{n}$-representation given the basis from Theorem \ref{basistheorem}. Namely as an element in $S_{n}$ need not send a critical cell to a critical cell, and $H_{k}\big(\text{cell}(n,2)\big)$ is not irreducible as an $S_{n}$-representation for $k\ge 1$ and $n\ge 3$. In Section \ref{newH1basis} we introduce a basis for $H_{1}\big(\text{cell}(n,2);\Q\big)$ distinct from the basis of Theorem \ref{basistheorem} that allows us to approach this problem in Sections \ref{newHnbasis} and \ref{putting it all together}.

\section{FI$_{d}$-modules}

Alpert proved that $H_{k}\big(\text{cell}(\bullet,2)\big)$ is a finitely generated FI$_{k+1}$-module \cite[Theorem 6.1]{alpert2020generalized}, and we will use this to decompose $H_{k}\big(\text{cell}(n,2);\Q\big)$ as a direct sum of induced $S_{n}$-representations. In this section we define the categories FI$_{d}$ and FI$_{d}$-mod. The second of these categories, studied in \cite{ramos2017generalized} and proven to be Noetherian in \cite{sam2017grobner}, will be play a central role in our calculations; for more, also see \cite{ramos2019configuration} and \cite{sam_snowden_2019}. Before we consider FI$_{d}$ and FI$_{d}$-mod, we recall the definition of the simpler categories FB and FB-mod.

\begin{defn}
Let \emph{FB} denote the category whose objects are finite sets and whose morphisms are bijective maps.
\end{defn}

Note that FB is equivalent to its skeleton which has one object $[n]:=\{1,\dots, n\}$ for each nonnegative integer. From now on we think of FB as this skeleton.

\begin{defn}
Given a commutative ring $R$, the category of \emph{FB-modules over $R$} has functors from FB to $R$-mod as objects, and has natural transformations between these functors as morphisms.
\end{defn}

Given an FB-module $U$, we write $U_{n}$ for $U([n])$. A priori there are no morphisms between $U_{n}$ and $U_{m}$, for $m\neq n$, since there are no morphisms between sets of different cardinalities in FB. Additionally, $R[S_{n}]$ acts on $U_{n}$ since $S_{n}\cong \text{End}_{\text{FB}}([n])$. Thus, we can think of an FB-module as a sequence of symmetric group representations with no relations between representations of symmetric groups of different cardinalities. If we allow injections between sets of different cardinality we get the category FI.

\begin{defn}
Let \emph{FI} denote the category whose objects are finite sets and whose morphisms are injective mappings.
\end{defn}

Just like in the case of FB, the category FI is equivalent to its skeleton with one object $[n]$ for each nonnegative integer. From now on we think of FI as this skeleton.

\begin{defn}
Given a commutative ring $R$, the category of \emph{FI-modules over $R$} has functors from FI to $R$-mod as objects, and has natural transformations between these functors as morphisms.
\end{defn}

As we did for FB, given an FI-module $V$, we write $V_{n}$ for $V([n])$. Like an FB-module, an FI-module can be thought of as a sequence of representations of the symmetric groups of various cardinalities, but with maps from $V_{n}$ to $V_{m}$ for $n<m$ arising from injections $[n]\hookrightarrow [m]$. For more information about the categories FB, FB-mod, FI, and FI-mod see \cite{church2015fi} or \cite{wilson2018introduction}.

FI does not make distinctions between elements in the codomain of a morphism that are not in the image. One way of adding more information to FI is to color these elements. Doing so yields a family of categories denoted FI$_{d}$.

\begin{defn}
For fixed $d\ge 1$ the category \emph{FI$_{d}$} has objects finite sets and morphisms pairs $(f,g)$ where $f$ is an injection of finite sets and $g$ is a $d$-coloring of the complement of the image of $f$, i.e., a map from the complement of the image of $f$ to $[d]:=\{1,\dots, d\}$.
\end{defn}

If the domain of $f$ is the codomain of $f'$ we can compose the morphisms $(f,g)$ and $(f', g')$ as follows

\[
(f,g)\circ(f', g')=(f\circ f',h),
\]
where
\[
h(x)=\begin{cases} g(x)&\text{if }x\notin \text{im}f\\g'(f^{-1}(x))&\text{otherwise.}\end{cases}
\]

See Figure \ref{FIdpic} for an example of the composition of two morphisms.

\begin{figure}[h!]
\centering
\includegraphics[width = 6 cm]{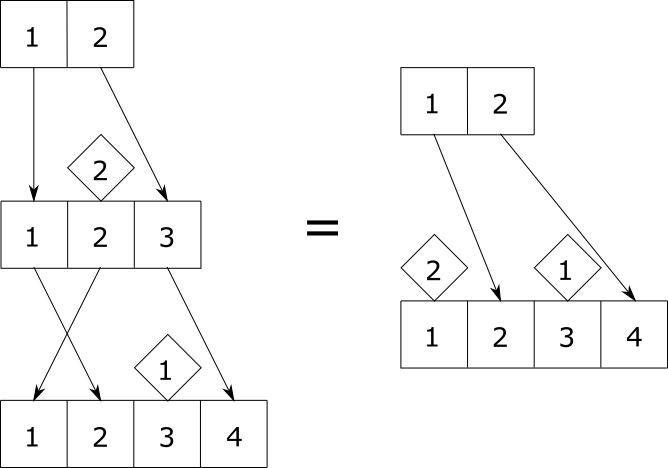}
\caption{An example of the composition of two morphisms in FI$_{2}$.}
\label{FIdpic}
\end{figure}

Just like for FB and FI, the category FI$_{d}$ is equivalent to its skeleton where there is one object $[n]$ for each nonnegative integer $n$. We think of FI$_{d}$ as this skeleton.

\begin{defn}
Given a commutative ring $R$, the category of \emph{FI$_{d}$-modules over $R$} has functors from FI$_{d}$ to $R$-mod as objects, and has natural transformations between these functors as morphisms.
\end{defn}

Given an FI$_{d}$-module $W$, we write $W_{n}$ for $W([n])$. While an FB-module can be thought of as a sequence of symmetric group representations without maps between them, for an FI$_{d}$-module there are numerous morphisms from representations of smaller symmetric groups to bigger ones. By forgetting the coloring an FI$_{d}$-mod can viewed as an FI-module.

\begin{defn}
Given an FI$_{d}$-module $W$, an \emph{FI$_{d}$-submodule} $W'$ of $W$ is a sequence of symmetric group representations $W'_{n}\subseteq W_{n}$ that is closed under the induced actions of the FI$_{d}$-morphisms.
\end{defn}

We define free FI$_{d}$-modules and discuss their structure as symmetric group representations, as this will be of great use in our attempt to decompose $H_{k}\big(\text{cell}(n,2);\Q\big)$ into induced representations.

\begin{defn}
For a nonnegative integer $m$, the \emph{free FI$_{d}$-module generated in degree $m$}, denoted $M^{FI_{d}}(m)$ is
\[
R\big[\text{Hom}_{\text{FI}_{d}}([m], \bullet)\big].
\]
\end{defn}

\begin{defn}
If $U_{m}$ is an $S_{m}$-representation, the \emph{free FI$_{d}$-module relative to $U_{m}$}, denoted $M^{FI_{d}}(U_{m}),$ is
\[
M^{FI_{d}}(U_{m}):=M^{FI_{d}}(m)\otimes_{R[S_{m}]}U_{m}.
\]

Since an FB-module is a sequence of symmetric group representations with no maps between them we are able to build an FI$_{d}$-module from an FB-module $U$:
\[
M^{FI_{d}}(U):=\bigoplus_{m\ge 0}M^{FI_{d}}(m)\otimes_{R[S_{m}]}U_{m}.
\]
is the free FI$_{d}$-module relative to $U$.
\end{defn}

\begin{defn}
Let $W$ be an FI$_{d}$-module. We say that $W$ is \emph{finitely generated in degree $\le k$} if for all $n$, $W_{n}$ is a quotient of $\oplus_{0\le j\le k}M^{FI_{d}}(W_{j})_{n}$ and $W_{j}$ is finitely generated as an $S_{j}$-representation for all $0\le j\le k$.
\end{defn}

\begin{prop}\label{dimofsimpleFId}
\[
\dim\big(M^{FI_{d}}(m)_{n}\big)=\frac{n!}{(n-m)!}d^{n-m}.
\]
\end{prop}

\begin{proof}
Since
\[
M^{FI_{d}}(m)_{n}=R\big[\text{Hom}_{\text{FI}_{d}}([m], [n])\big]
\]
it suffices to count the number of injections from $[m]$ to $[n]$. There are $n$ choices of where to send $1\in[m]$, $n-1$ choices of where to $2\in[n]$ given that it cannot be sent to where $1$ was sent, so on and so forth. Of the remaining $n-m$ elements in $[n]$ not in the image of a given injection each can be colored one of $d$ colors. Thus,
\[
\dim\big(M^{FI_{d}}(m)_{n}\big)=\frac{n!}{(n-m)!}d^{n-m}.
\]
\end{proof}

\begin{prop}\label{buidanFIdfromanFB}
\[
\dim\big(M^{FI_{d}}(U_{m})_{n}\big)=\dim(U_{m})\binom{n}{m}d^{n-m}.
\]
\end{prop}

\begin{proof}
By Proposition \ref{dimofsimpleFId} 
\[
\dim(M^{FI_{d}}(m)_{n}\big)=\frac{n!}{(n-m)!}d^{n-m}.
\]
Since $M^{FI_{d}}(U_{m})_{n}=M^{FI_{d}}(m)_{n}\otimes_{R[S_{m}]}U_{m}$,
\[
\dim\big(M^{FI_{d}}(m)_{n}\otimes_{R[S_{m}]}U_{m}\big)=\dim\left(M^{FI_{d}}(m)_{n}\right)\frac{\dim(U_{m})}{\dim\big(R[S_{m}]\big)}.
\]
This simplifies to
\[
\frac{n!}{(n-m)!}d^{n-m}\frac{\dim(U_{m})}{m!}=\dim(U_{m})\binom{n}{m}d^{n-m}.
\]
\end{proof}

Now that we know the dimension of a free FI$_{d}$-module in degree $n$, we want to determine how it decomposes into irreducible representations of $S_{n}$.

\begin{defn}
A \emph{$h$-composition} of $n$ is an $h$-tuple $a:=(a_{1},\dots, a_{h})$ of nonnegative integers such that $\sum_{i=1}^{h}a_{i}=n$. Additionally, we write $|a|=n$.
\end{defn}

If $a=(a_{1},\dots, a_{h})$ we set $S_{a}:=S_{a_{1}}\times \cdots \times S_{a_{h}}\subseteq S_{|a|}\cong S_{n}$.

\begin{prop}\label{freeFId}
(Ramos \cite[Proposition 3.4]{ramos2017generalized}) Let $U_{m}$ be an $S_{m}$-representation. Then, 
\[
M^{FI_{d}}(U_{m})_{n}=\bigoplus_{a=(a_{1},\dots, a_{d}), |a|=n-m}\emph{Ind}^{S_{n}}_{S_{m}\times S_{a}}W_{m}\boxtimes R,
\]
where $R$ is the trivial representation of $S_{a}$.
\end{prop}

When is an FB-module $U$ with maps $[i_{n,j}]:U_{n}\to U_{n+1}$ for all $n$ and $1\le j\le d$ an FI$_{d}$-module? The following definitions and lemmas give us necessary and sufficient conditions on the maps $[i_{n,j}]$ for this to be the case.

\begin{defn}
Let $U$ be an FB-module, and for all $n$ and all $1\le j\le d$ let there exist maps $[i_{n,j}]:U_{n}\to U_{n+1}$. We call the map $[i_{n,j}]$ the \emph{$j^{\text{th}}$ high-insertion map}.
\end{defn}

\begin{defn}
Let $U$ be an FB-module with $d$ high-insertion maps $[i_{n,j}]$ for all $n$. We say \emph{high-insertion maps commute with permutations} if for every $j$, every $n$, and every $\sigma\in S_{n}$, we have
\[
[i_{n,j}]\circ[\sigma]=[\tilde{\sigma}]\circ[i_{n,j}]
\]
where $\tilde{\sigma}\in S_{n+1}$ is the image of $\sigma$ in $S_{n+1}$ under the standard inclusion $[n]\hookrightarrow [n+1]$.
\end{defn}

\begin{defn}
Let $U$ be an FB-module with $d$ high-insertion maps $[i_{n,j}]$ for all $n$. We say \emph{insertions are unordered} if for every pair $j$, $l$ and every $n$, we have the following relation on maps from $U_{n}$ to $U_{n+2}$:
\[
[(n+1\, n+2)]\circ[i_{n+1,j}]\circ[i_{n,l}]=[i_{n+1,l}]\circ[i_{n,j}].
\]
\end{defn}

\begin{lem}
(Alpert \cite[Lemma 5.1]{alpert2020generalized}) Suppose we have an FB-module $U$ and $d$ high-insertion maps $[i_{n,j}]$ from $U_{n}$ to $U_{n+1}$ for all $n$. If the high-insertion maps commute with permutations and insertions are unordered, then $U$, along with these high-insertion maps, forms an FI$_{d}$-module. 
\end{lem}

\begin{prop}\label{makesmallerFId}
If $W$ is an FI$_{d}$-module, then it is also an FI$_{d-1}$-module.
\end{prop}




The symmetric group action on $\text{cell}(\bullet,2)$ extends to an action on $H_{k}\big(\text{cell}(\bullet,2)\big)$ for all $k$, meaning that the latter is an FB-module. Thus, if we can define $d$ high-insertion maps and check if that they commute with permutations and are unordered, it follows that $H_{k}\big(\text{cell}(\bullet,2)\big)$ is an FI$_{d}$-module. Alpert was able to do that, showing that there exist $k+1$ high-insertion maps in $H_{k}\big(\text{cell}(\bullet,2)\big)$ and that they have the desired properties \cite[Theorem 6.1]{alpert2020generalized}.

\begin{defn}
Let $e$ be a critical cell written as a unique concatenation product of images of the cells $1$, $1\,2$, and $1|3\,2$ under order-preserving injections. Each image of $1\,2$ and $1|3\,2$ is said to be a \emph{barrier}.
\end{defn}

\begin{defn}
Consider $H_{k}\big(\text{cell}(n,2)\big)$, and let $0\le j\le k$. Given a critical cell $e$ of $\text{cell}(n,2)$, we define the \emph{$j^{\text{th}}$ high-insertion map $\iota_{j}(e)$} to be the map that arises when we insert a block containing the label  $n+1$, right after the $j^{\text{th}}$ barrier of $e$, or as the first block, if $j=0$. Observe that the resulting cell is critical. Thus, these maps give rise to maps on homology $[\iota_{j}]:H_{k}\big(\text{cell}(n,2)\big)\to H_{k}\big(\text{cell}(n+1,2)\big)$. Given a homology class $\sum a_{e}z(e)\in H_{k}\big(\text{cell}(n,2)\big)$, we set 
\[
[\iota_{j}]\sum a_{e}z(e):=\sum a_{e}z\big(\iota_{j}(e)\big)\in H_{k}\big(\text{cell}(n+1,2)\big).
\]
\end{defn}

\begin{exam}
Consider the critical cell $2|5\,3|1|4\,6$. We have that $z(2|5\,3|1|4\,6)\in H_{2}\big(\text{cell}(6,2)\big)$, and 
\[
\iota_{0}(2|5\,3|1|4\,6)=7|2|5\,3|1|4\,6\mbox{,	}\iota_{1}(2|5\,3|1|4\,6)=2|5\,3|7|1|4\,6\mbox{, and }\iota_{2}(2|5\,3|1|4\,6)=2|5\,3|1|4\,6|7.
\]
These maps induce the maps 
\[
[\iota_{0}]z(2|5\,3|1|4\,6)=z\big(\iota_{0}(2|5\,3|1|4\,6)\big)=z(7|2|5\,3|1|4\,6)=z(7)|z(2|5\,3)|z(1)|z(4\,6),
\]
\[
[\iota_{1}]z(2|5\,3|1|4\,6)=z\big(\iota_{1}(2|5\,3|1|4\,6)\big)=z(2|5\,3|7|1|4\,6)=z(2|5\,3)|z(7)|z(1)|z(4\,6),
\]
and 
\[
[\iota_{2}]z(2|5\,3|1|4\,6)=z\big(\iota_{2}(2|5\,3|1|4\,6)\big)=z(2|5\,3|1|4\,6|7)=z(2|5\,3)|z(1)|z(4\,6)|z(7).
\]
\end{exam}

Alpert proved that these high-insertion maps commute with permutations and are unordered.

\begin{lem}
(Alpert \cite[Lemma 6.4]{alpert2020generalized}) For each $k$ and $n$, the $S_{n}$-actions and high-insertion maps on the homology groups $H_{k}(\text{cell}(n,2))$ have the properties high-insertion maps commute with permutations and high-insertions are unordered.
\end{lem}

Additionally, Alpert showed that singletons not immediately before a two element follower commute.

\begin{lem}\label{switchsingles}
(Alpert \cite[Lemma 6.3]{alpert2020generalized}) Let $z_{1}$ and $z_{2}$ be injected cycles whose symbols are on disjoint sets, and let $p$ and $q$ be labels not appearing in $z_{1}$ and $z_{2}$. Then the concatenation products of cycles $z_{1}|p|q|z_{2}$ and $z_{1}|q|p|z_{2}$ are homologous.
\end{lem}

With those two lemmas and a few other results Alpert was able to prove that $H_{k}\big(\text{cell}(\bullet,2)\big)$ is a finitely generated FI$_{k+1}$-module.

\begin{thm}\label{AlpertsBIGTheorem}
(Alpert \cite[Theorem 6.1]{alpert2020generalized}) For any $k$, the homology groups $H_{k}\big(\text{cell}(\bullet,2)\big)\cong H_{k}\big(\text{Conf}(\bullet,2)\big)$ form a finitely generated $FI_{k+1}$-module over $\Z$.
\end{thm}

It follows that $H_{k}(\text{cell}(\bullet,2))$ has quite a bit of structure as a sequence of symmetric group representations. Still, Theorem \ref{AlpertsBIGTheorem} and the results around it do not provide a way to decompose $H_{k}\bullet(\text{cell}(n,2)\big)$ into irreducible symmetric group representations, as they do not describe the relations that arise. In the next section we take up this task, by considering rational homology. Since, $H_{k}\big(\text{cell}(n,2)\big)$ is torsion free, it follows that \ref{AlpertsBIGTheorem} proves

\begin{cor}\label{AlpertsBIGTheoremcor}
For any $k$, the homology groups $H_{k}\big(\text{cell}(\bullet,2);\Q\big)\cong H_{k}\big(\text{Conf}(\bullet,2);\Q\big)$ form a finitely generated $FI_{k+1}$-module over $\Q$.
\end{cor}

\section{A basis for $H_{1}\big(\text{cell}(n,2);\Q\big)$} \label{newH1basis}

An insurmountable obstacle in decomposing $H_{k}\big(\text{cell}(n,2)\big)$ into a direct sum of irreducible $S_{n}$-representations arises in $H_{1}(\text{cell}(3,2))$ and remains for all nontrivial $H_{k}\big(\text{cell}(n,2)\big)$ for $k\ge 1$ and $n\ge 3$: these groups are not indecomposable as $S_{n}$-representations \cite[Theorem 8.3]{alpert2021configuration1}. Alpert and Manin give a presentation for $H_{*}(\text{Conf}(\bullet,2))$ as a twisted algebra taking this into account in \cite[Theorem 8.3]{alpert2021configuration1}, but the relations in this presentation make explicitly stating the $S_{n}$-representation structure computationally challenging for large $n$, as they are concentrated in small degree. This problem goes away if we consider homology with rational coefficients, as $H_{k}\big(\text{cell}(n,2);\Q\big)$ is a semi-simple $S_{n}$-representation. Since $H_{k}\big(\text{cell}(n,2)\big)$ is torsion free, it is of the same dimension as $H_{k}\big(\text{cell}(n,2);\Q\big)$. While the basis from Theorem \ref{basistheorem} can be thought of as a basis for the rational homology groups, it does not lend itself to decomposing $H_{k}\big(\text{cell}(n,2);\Q\big)$ into a direct sum of irreducible $S_{n}$-representations. In this section we give a new basis for $H_{1}\big(\text{cell}(n,2);\Q\big)$ that has several nice representation theoretic properties that will aid in our efforts to decompose $H_{k}\big(\text{cell}(n,2);\Q\big)$ into a direct sum of induced $S_{n}$-representations.

First, we give an example of the computational problems that arise when restrict ourselves to homology with integer coefficients and the basis of Theorem \ref{basistheorem}. Consider $H_{1}\big(\text{cell}(3,2)\big)$. The basis given by Theorem \ref{basistheorem} consists of $z(1|3\,2)$, $z(1|2\,3)$, $z(2\,3|1)$, $z(2|1\,3)$, $z(1\,3|2)$, $z(3|1\,2)$, and $z(1\,2|3)$. The symmetric group $S_{3}$ stabilizes the set consisting of the last six elements. As a result, their span is a representation of $S_{3}$; unfortunatley, $S_{3}$ does not send the span of $z(1|3\,2)$ to itself. Instead, nontrivial elements of $S_{3}$ send $z(1|3\,2)$ to nontrivial linear combinations of all seven basis elements. Thus, even if $H_{1}\big(\text{cell}(3,2)\big)$ was semi-simple, the task of decomposing $H_{1}\big(\text{cell}(3,2);\Q\big)$ into irreducible $S_{3}$-representations is harder than necessary given the basis of Theorem \ref{basistheorem}. These problems remain when we increase the homological degree $k$ and the number of labels $n$. Thus, we turn to rational coefficients as $H_{k}\big(\text{cell}(p,2);\Q\big)$ is a semi-simple $S_{n}$-representation, and seek a new basis for homology.

We begin this section by separating the basis elements for $H_{1}\big(\text{cell}(n,2)\big)$ described in Theorem \ref{basistheorem} into two classes.

\begin{defn}\label{Adef}
For $n\ge 3$ and $1\le m\le n-2$ let \emph{$A_{n,m}$} denote the set of basis elements for $H_{1}\big(\text{cell}(n,2);\Q\big)$ described in Theorem \ref{basistheorem} of the form $z(e)$ where $e$ is a critical cell of the form $i_{1}|\cdots|i_{m}|i_{m+1}\,i_{m+2}|i_{m+3}|\cdots|i_{n}$, with $i_{j}\in \{1,\cdots, n\}$, such that, for all $j\neq k$, $i_{j}\neq i_{k}$, $i_{1}>\cdots>i_{m}$, $i_{m+3}>\cdots>i_{n}$, and $i_{m+1}>i_{m+2}>i_{m}$. Set $A_{n}:=\bigcup_{1\le m\le n-2}A_{n, m}$.
\end{defn}

The set $A_{n}$ is the set of basis elements for $H_{1}\big(\text{cell}(n,2); \Q\big)$ from Theorem \ref{basistheorem} arising from left and right high-insertion, i.e., $[i_{0}]$ and $[i_{1}]$, on $z(1|3\,2)$ and a partial action of the symmetric group. Note that $A_{n}$ is not preserved by every element of $S_{n}$ action, e.g., $(123)\in S_{3}$ does not send $z(1|3\,2)$, the only element of $A_{3}$, to a multiple of itself. It follows from Definition \ref{Adef} that the label $1$ appears immediately to the left of the block of size two in the representative cell of an element of $A_{n,1}$.

\begin{exam}
The set $A_{3,1}$ has only one element namely, $z(1|3\,2)$. This is the blue, inner loop oriented clockwise in Figure \ref{1cycleincell(3,2)}. $A_{4,1}$ consists of three elements $z(2|1|4\,3)$, $z(3|1|4\,2)$, and $z(4|1|3\,2)$. 
\end{exam}

\begin{defn}
For $n\ge 2$ and $0\le m\le n-2$ let \emph{$B_{n, m}$} denote the set of basis elements for $H_{1}\big(\text{cell}(n,2);\Q\big)$ described in Theorem \ref{basistheorem} of the form $z(e)$ where $e$ is a critical cell of the form $i_{1}|\cdots|i_{m}|i_{m+1}\,i_{m+2}|i_{m+3}|\cdots|i_{n}$, with $i_{j}\in \{1,\cdots, n\}$, for all $j\neq k$, $i_{j}\neq i_{k}$, $i_{1}>\cdots>i_{m}$, $i_{m+3}>\cdots>i_{n}$, and $i_{m+2}>i_{m+1}$. Set $B_{n}:=\bigcup_{0\le m\le n-2}B_{n,m}$.
\end{defn}

These are the basis elements for $H_{1}\big(\text{cell}(n,2); \Q\big)$ from Theorem \ref{basistheorem} arising from left and right high-insertion, i.e., $[i_{0}]$ and $[i_{1}]$, on $z(1\,2)$ and the action of the symmetric group.

\begin{exam}
The classes in $B_{3}$ are the small loops in Figure \ref{cell(3,2)pic}, i.e., $z(1|2\,3)$, $z(2\,3|1)$, $z(2|1\,3)$, $z(1\,3|2)$, $z(3|1\,2)$, and $z(1\,2|3)$.
\end{exam}

Note that $A_{n}\cup B_{n}$ is the basis for $H_{1}\big(\text{cell}(n,2); \Q\big)$ from Theorem \ref{basistheorem}.

The symmetric group action does not send the elements of $B_{n}$ to formal sums of elements of $B_{n}$. Fortunately, the symmetric group sends the elements of $B_{n}$ to elements homologous to the elements of $B_{n}$, i.e., $S_{n}$ stabilizes this set of homology classes represented by the span of the elements of $B_{n}$.

\begin{prop}\label{BnmisSnhomologygood}
Let $n\ge 2$ and $m$ be such that $0\le m\le n-2$. Then, $S_{n}$ stabilizes the set of homology classes generated by single elements of $B_{n,m}$. Moreover, the $S_{n}$-action is transitive on this set of homology classes.
\end{prop}

\begin{proof}
We can write elements of $B_{n,m}$ as concatenation products:
\[
z(i_{1}|\cdots|i_{m}|i_{m+1}\,i_{m+2}|i_{m+3}|\cdots|i_{n})=z(i_{1})|\cdots|z(i_{m})|z(i_{m+1}\,i_{m+2})|z(i_{m+3})|\cdots|z(i_{n}).
\]
It follows that 
\begin{align*}
z(i_{1}|\cdots|i_{m}|i_{m+1}\,i_{m+2}|i_{m+3}|\cdots|i_{n})
=i_{1}|\cdots|i_{m}|z(i_{m+1}\,i_{m+2})|i_{m+3}|\cdots|i_{n}\\
=i_{1}|\cdots|i_{m}|i_{m+1}\,i_{m+2}|i_{m+3}|\cdots|i_{n}\,+\,i_{1}|\cdots|i_{m}|i_{m+2}\,i_{m+1}|i_{m+3}|\cdots|i_{n}.
\end{align*}
Given $\sigma$, an arbitrary element of $S_{n}$,
\begin{multline*}
\sigma\big(z(i_{1}|\cdots|i_{m}|i_{m+1}\,i_{m+2}|i_{m+3}|\cdots|i_{n})\big)
=\sigma(i_{1})|\cdots|\sigma(i_{m})|\sigma(i_{m+1})\,\sigma(i_{m+2})|\sigma(i_{m+3})|\cdots|\sigma(i_{n})\,\\+\,\sigma(i_{1})|\cdots|\sigma(i_{m})|\sigma(i_{m+2})\,\sigma(i_{m+1})|\sigma(i_{m+3})|\cdots|\sigma(i_{n}).
\end{multline*}
This simplifies to
\[
\sigma(i_{1})|\cdots|\sigma(i_{m})|\big(\sigma(i_{m+1})\,\sigma(i_{m+2})\,+\,\sigma(i_{m+2})\,\sigma(i_{m+1})\big)|\sigma(i_{m+3})|\cdots|\sigma(i_{n}).
\]
It might not be the case that $\sigma(i_{1})>\cdots>\sigma(i_{m})$ or $\sigma(i_{m+3})>\cdots>\sigma(i_{n})$. Lemma \ref{switchsingles} proves that this cycle is homologous to the cycle where $\sigma(i_{1})>\cdots>\sigma(i_{m})$, and $\sigma(i_{m+3})>\cdots>\sigma(i_{n})$. Additionally,
\[
\sigma(i_{m+1})\,\sigma(i_{m+2})+\sigma(i_{m+2})\,\sigma(i_{m+1})=\sigma(i_{m+2})\,\sigma(i_{m+1})+\sigma(i_{m+1})\,\sigma(i_{m+2})
\]
Therefore,
\[
\sigma\big(z(i_{1}|\cdots|i_{m}|i_{m+1}\,i_{m+2}|i_{m+3}|\cdots|i_{n})\big)
\]
is homologous to 
\[
z(j_{1}|\cdots|j_{m}|j_{m+1}\,j_{m+2}|j_{m+3}|\cdots|j_{n}),
\]
such $j_{k}\in\{1,\dots, n\}$ for all $k$, for $1\le k\le m$ there is some $1\le r\le m$ such that $j_{k}=\sigma(i_{r})$, for $m+3\le k\le n$ there is some $m+3\le r\le n$ such that $j_{k}=\sigma(i_{r})$, $j_{j}\neq j_{l}$ if $k\neq l$, $j_{1}>\cdots>j_{m}$, $j_{m+1}<j_{m+2}$, $j_{m+1}$ is the lesser of $\sigma(i_{m+1})$ and $\sigma(i_{m+2})$ and $j_{m+2}$ is the greater of that pair, and $j_{m+3}>\cdots>j_{n}$. Therefore, $S_{n}$ stabilizes the set of homology classes generated by single elements of $B_{n,m}$.

Next we prove that the $S_{n}$-action is transitive on the set of homology classes  generated single elements of $B_{n,m}$. Let $z(i_{1}|\cdots|i_{m}|i_{m+1}\,i_{m+2}|i_{m+3}|\cdots|i_{n})$ and $z(j_{1}|\cdots|j_{m}|j_{m+1}\,j_{m+2}|j_{m+3}|\cdots|j_{n})$ be two elements of $B_{n,m}$. The permutation $\sigma\in S_{n}$ that maps $i_{k}$ to $j_{k}$ for all $k\in[n]$, maps the first element of $B_{n,m}$ to the second element. Since our choice of elements was arbitrary, any element of $B_{n,m}$ can be mapped to any other element of $B_{n,m}$, so the $S_{n}$-action on the set of homology classes  generated single elements of $B_{n,m}$ is transitive.
\end{proof}

\begin{prop}\label{BnisSnhomologygood}
For $n\ge 2$, $S_{n}$ stabilizes the set of homology classes generated by single elements of $B_{n}$.
\end{prop}

\begin{proof}
This follows immediately from Proposition \ref{BnmisSnhomologygood} and the fact that
\[
B_{n}=\bigcup_{0\le m\le n-2}B_{n,m}.
\]
\end{proof}

We use the previous two propositions to show that the FI$_{2}$-module generated by $z(1\,2)$ and the high-insertion maps $[\iota_{0}]$ and $[\iota_{1}]$ is a free FI$_{2}$-submodule of $H_{1}\big(\text{cell}(\bullet,2);\Q\big)$.

\begin{prop}\label{BnisabasisforthefreeFI2mod}
For $n\ge 2$ the elements of $B_{n}$ form a basis for the $n^{\text{th}}$-degree term of the FI$_{2}$-submodule of $H_{1}\big(\text{cell}(\bullet,2);\Q\big)$ generated by the span of $z(1\,2)$ and the high-insertion maps $[\iota_{0}]$ and $[\iota_{1}]$. Moreover, this FI$_{2}$-submodule is free, and it is isomorphic to $M^{FI_{2}}(Triv_{2})$, where $Triv_{2}$ is the trivial representation of $S_{2}$.
\end{prop}

\begin{proof}
First, we show that every element of $B_{n}$ is in the FI$_{2}$-submodule of $H_{1}\big(\text{cell}(\bullet,2);\Q\big)$ generated by  $z(1\,2)$ and the high-insertion maps $[\iota_{0}]$ and $[\iota_{1}]$. By definition of an FI$_{2}$-module, every homology class of the form 
\[
\sigma\big(z(i_{1})|\cdots|z(i_{m})|z(1\,2)|z(i_{m+1})|\cdots|z(i_{n-2})\big),
\]
where $0\le m\le n-2$, $i_{j}\in \{3,\dots, n\}$, $i_{j}\neq i_{k}$ for $j\neq k$, $i_{1}>\cdots>i_{m}$, $i_{m+1}>\cdots>i_{n-2}$, and $\sigma\in S_{n}$ is in the $n^{\text{th}}$-degree of the FI$_{2}$-submodule of $H_{1}\big(\text{cell}(\bullet,2);\Q\big)$ generated by $z(1\,2)$ and $[\iota_{0}]$ and $[\iota_{1}]$. Note that 
 \[
 \sigma\big(z(i_{1})|\cdots|z(i_{m})|z(1\,2)|z(i_{m+1})|\cdots|z(i_{n-2})\big)=\sigma\big(z(i_{1}|\cdots|i_{m}|1\,2|i_{m+1}|\cdots|i_{n-2})\big),
 \]
and that $z(i_{1}|\cdots|i_{m}|1\,2|i_{m+1}|\cdots|i_{n-2})$ is an element of $B_{n,m}$. By Proposition \ref{BnmisSnhomologygood} the set of homology classes generated by the elements of $B_{n,m}$ is stable under the $S_{n}$-action and this action is transitive on these classes, so the set of homology classes generated by the elements of $B_{n,m}$ is in this FI$_{2}$-submodule, and $B_{n,m}$ gives a representative basis for this space. Since $m$ can range from $0$ to $n-2$, it follows that all the homology classes generated by single elements of $B_{n}$ are in this FI$_{2}$-submodule. It follows that the span of the elements of $B_{n}$ is in this FI$_{2}$-submodule.

Note that $S_{2}$ acts as the trivial representation on the $\Q$-span of $z(1\,2)$ in $H_{1}\big(\text{cell}(2,2);\Q\big)$. Thus, the FI$_{2}$-submodule of $H_{1}\big(\text{cell}(\bullet,2);\Q\big)$ generated by $z(1\,2)$ and $[\iota_{0}]$ and $[\iota_{1}]$ is isomorphic to FI$_{2}$-submodule of $M^{\text{FI}_{2}}(Triv_{2})$.

Next we show that this FI$_{2}$-submodule is free by showing that 
\[
|B_{n}|=\dim(M^{FI_{2}}(Triv_{2})_{n}).
\]

There are $\binom{n}{2}$ ways of choosing the two labels of the block of size $2$ from $[n]$ in a element of $B_{n}$. There are $2^{n-2}$ ways of placing the remaining elements of $n-2$ elements either to the left or right of this block such that the elements of $[n]$ on the same side of the block of size $2$ are in decreasing order. Thus 
\[
|B_{n}|=\binom{n}{2}2^{n-2}.
\]

By Theorem \ref{basistheorem} the elements of $B_{n}$ are linearly independent in $H_{1}\big(\text{cell}(n,2);\Q\big)$. Thus the $\Q$-span of the elements in $B_{n}$ in $H_{1}\big(\text{cell}(n,2);\Q\big)$ is $\binom{n}{2}2^{n-2}$-dimensional. Since the FI$_{2}$-submodule of $H_{1}\big(\text{cell}(\bullet,2);\Q\big)$ generated by $z(1\,2)$ and the maps $[\iota_{0}]$ and $[\iota_{1}]$ contains the $\Q$-span of $B_{n}$, it follows that this FI$_{2}$-module is at least $\binom{n}{2}2^{n-2}$ dimensional. 

By Proposition \ref{buidanFIdfromanFB}
\[
\dim\big(M^{\text{FI}_{2}}(Triv_{2})_{n}\big)=\dim(\text{Triv}_{2})\binom{n}{2}2^{n-2}=\binom{n}{2}2^{n-2}.
\]

Since the $n^{\text{th}}$-degree of the FI$_{2}$-module generated by $z(1\,2)$ and $[\iota_{0}]$ and $[\iota_{1}]$ is at least $\binom{n}{2}2^{n-2}$ dimensional, it follows that the FI$_{2}$-submodule is isomorphic to $M^{FI_{2}}(Triv_{2})$ and is therefore free. Moreover, since $B_{n}$ spans a $\binom{n}{2}2^{n-2}=\dim\big(M^{\text{FI}_{2}}(Triv_{2})_{n}\big)$-dimensional subspace of the $n^{\text{th}}$-degree term of this free FI$_{2}$-submodule, it follows that $B_{n}$ is a basis for the $n^{\text{th}}$-degree term of the free FI$_{2}$-submodule of $H_{1}\big(\text{cell}(\bullet,2);\Q\big)$ generated by the $\Q$-span of $z(1\,2)$ and the high-insertion maps $[\iota_{0}]$ and $[\iota_{1}]$.
\end{proof}

\begin{defn}
Let $z(e)$ be an element the basis for $H_{1}\big(\text{cell}(n,2);\Q\big)$ from Theorem \ref{basistheorem}. If $z(e)\in A_{n,m}$, we set $L(z(e)):=2m$, and if $z(e)\in B_{n,m}$, we set $L(z(e)):=2m-1$. Given a cycle $z$ in $H_{1}\big(\text{cell}(n,2);\Q\big)$ of the form $z=\sum a_{e}z(e)$ where all $a_{e}\neq 0$ we say that the \emph{leading value} of $z$ is $L(z):=\max(L(z(e)))$.
\end{defn}

As stated earlier, the $S_{3}$-action does not send $z(1|3\,2)$ to another element of the basis from Theorem \ref{basistheorem}. This makes it a challenge to use that basis to compute the symmetric group representation structure of $H_{1}\big(\text{cell}(n,2);\Q\big)$. To overcome this challenge, we describe a new family of bases for $H_{1}\big(\text{cell}(n,2);\Q\big)$ that have nicer representation theoretic properties. We do this by replacing the sets $A_{n,m}$ with sets $A'_{n,m}$ such that the $\Q$-span of the elements of $A'_{n,m}$ is preserved under the $S_{n}$-action.

\begin{thm}\label{betterbasisforH1}
For $n\ge 3$ and $1\le m\le n-2$ there are finite sets $A'_{n,m}$ consisting of cycles in $H_{1}\big(\text{cell}(n,2);\Q\big)$ such that 
\begin{enumerate}
\item The elements of $A'_{n,m}$ are linearly independent and $|A'_{n,m}|=|A_{n,m}|$,
\item The $\Q$-span of $B_{n,0}\bigcup_{1\le l\le m}\big(A'_{n,l}\cup B_{n,l}\big)$ in $H_{1}\big(\text{cell}(n,2);\Q\big)$ equals the $\Q$-span of $B_{n,0}\bigcup_{1\le l\le m}\big(A_{n,l}\cup B_{n,l}\big)$ in $H_{1}\big(\text{cell}(n,2);\Q\big)$,
\item The $\Q$-span of $A'_{n,m}$ is the $n^{\text{th}}$-degree of the FI$_{1}$-submodule of $H_{1}\big(\text{cell}(\bullet, 2);\Q\big)$ generated by the $\Q$-span of $A'_{m+2,m}$ and the high-insertion map $[\iota_{1}]$ and this FI$_{1}$-submodule is free, and
\item As a representation of $S_{m+2}$ the $\Q$-span of $A'_{m+2, m}$ is isomorphic to $\bigwedge^{2}Std_{m+2}$ where $Std_{m+2}$ is the standard $S_{m+2}$-representation, and this is a subrepresentation of the $(m+2)^{\text{th}}$ term of the FI$_{1}$-submodule of $H_{1}\big(\text{cell}(\bullet, 2);\Q\big)$ generated by the $\Q$-span of $A'_{3,1}$ and the high-insertion map $[\iota_{0}]$.
\end{enumerate}
\end{thm}

Note that this theorem claims that there exists high-insertion maps $[\iota_{0}]$ and $[\iota_{1}]$ acting on the elements of $A'_{n,m}$.  This is allowed, and makes sense, due to Theorem \ref{basistheorem}, which provides a basis for $H_{1}\big(\text{cell}(n,2);\Q\big)$ such that every basis element has $1$ barrier. Thus, every element in $H_{1}\big(\text{cell}(n,2);\Q\big)$ is a linear combination of these basis elements, and more specifically every element in $A'_{m,n}$ is a linear combination of these basis elements. Recall that if we are given some linear combination of these basis elements, say $z=\sum a_{e}z(e)$, then 
\[
[\iota_{0}]z=[\iota_{0}]\sum a_{e}z(e)=\sum a_{e}[\iota_{0}]z(e)=\sum a_{e}z(\iota_{0}e)=\sum a_{e}z(n+1|e)
\]
\[
=\sum a_{e}z(n+1)|z(e)=z(n+1)|\sum a_{e}z(e)=z(n+1)|z.
\]
Similarly, we have that 
\[
[\iota_{1}]z=z|z(n+1)
\]
as homology classes after applying Lemma \ref{switchsingles}. Thus, it makes sense to consider the high-insertion maps $[\iota_{0}]$ and $[\iota_{1}]$ on the elements of our sets $A'_{n,m}$.

Before we begin the proof we make a quick remark on its structure. The proof is an induction argument. We first show that there is a set $A'_{3,1}$ with the desired properties. Given the existence of this set we show that there exist sets $A'_{n,1}$ with the desired properties arising from $A'_{3,1}$. We then assume the existence of sets $A'_{n,m}$ for all $m\le k$ satisfying the theorem state that arise from the previous sets. We use this to show that that there exists a set $A'_{k+3,k+1}$ with the desired properties that ``plays nicely" with the previous sets. We finally show that there exists sets $A'_{n,k+1}$ for all $n\ge k+3$ with the desired properties that such that all sets constructed have the desired properties. See Figure \ref{proofstructure} for a diagram.

\begin{figure}[H]\label{proofstructure}
\centering
\begin{tikzpicture} \footnotesize
  \matrix (m) [matrix of math nodes, nodes in empty cells, nodes={minimum width=3ex, minimum height=5ex, outer sep=2ex}, column sep=3ex, row sep=3ex]{
  &  & &  &&&&&\\
&4   &  \emptyset  &\emptyset &  \emptyset   &\emptyset &\emptyset&A'_{6,4}&A'_{7,4}&\\  
&3   &  \emptyset   & \emptyset   & \emptyset & \emptyset & A'_{5,3} &A'_{6,3} &A'_{7,3}&\\          
m&2     &  \emptyset   & \emptyset   & \emptyset& A'_{4,2} & A'_{5,2}&A'_{6,2}& A'_{7,2}&\\             
&1    &  \emptyset  & \emptyset & A'_{3,1}  & A'_{4,1}  & A'_{5,1}&A'_{6,1}&A'_{7,1}&\\       
&    \strut &   1  &  2  &  3  & 4  &5&6&7&\\
& &  &  &   & n  &&&&\\
}; 

\draw[-stealth, red] (m-5-5) -- (m-5-10) [midway,above] node [midway,below] {} ;
\draw[-stealth, red] (m-4-6) -- (m-4-10) [midway,above] node [midway,below] {} ;
\draw[-stealth, red] (m-3-7) -- (m-3-10) [midway,above] node [midway,below] {} ;
\draw[-stealth, red] (m-2-8) -- (m-2-10) [midway,above] node [midway,below] {} ;

 \draw[thick] (m-2-2.east) -- (m-6-2.east) ;
 \draw[thick] (m-6-2.north) -- (m-6-9.north east) ;
\end{tikzpicture}
\caption{The induction argument of Theorem \ref{betterbasisforH1}. We first prove the existence of a set $A'_{3,1}$ and use this to construct the sets in the bottom row. We then induct, assuming we can construct all the sets up through the $k^{\text{th}}$-row. We use this assumption to construct the first nontrivial set in the $(k+1)^{\text{th}}$ row $A'_{k+3,k+1}$, and use this set to show that we can construct all the sets in that row such that all sets constructed have the desired properties.}
\label{e2n8}
\end{figure}
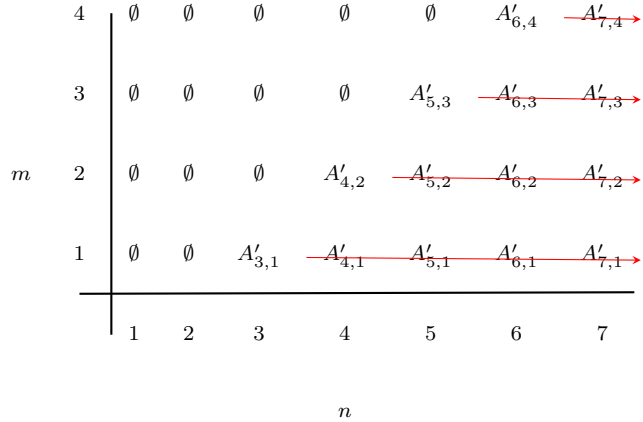

\begin{proof}
We begin by showing there is a set $A'_{3,1}$ with the desired properties. Let $A'_{3,1}=\{z'(1|3\,2)\}$, where
\[
z'(1|3\,2)=\frac{1}{2}\big(1|3\,2-1|2\,3+3\, 1|2-1\,3|2+3|2\,1-3|1\, 2-3\,2|1+2\,3|1-2|3\,1+2|1\,3-2\,1|3+1\,2|3\big).
\]
This is the average of the large inner (blue) and outer (yellow) loops of Figure \ref{1cycleincell(3,2)} if they both have a clockwise orientation, and we can write $z'(1|3\,2)$ as a sum of elements from the basis of Theorem \ref{basistheorem} as follows
\[
z'(1|3\,2)=z(1|3\,2)+\frac{1}{2}\big(-z(1|2\,3)+z(2|1\,3)-z(3|1\,2)+z(2\,3|1)-z(1\,3|2)+z(1\,2|3)\big).
\]
From this it follows that $z'(1|3\,2)$ is nontrivial and $|A'_{3,1}|=|A_{3,1}|$, so $A'_{3,1}$ satisfies property $1$. It also follows that the $\Q$-span of $B_{3}\cup A'_{3,1}$ equals the $\Q$-span of $B_{3}\cup A_{3,1}$, so $A'_{3,1}$ satisfies property $2$. 

One can check that transpositions act on $z'(1|3\,2)$ by multiplication by $-1$ and that that $3$-cycles act by multiplication by $1$. Therefore, the $\Q$-span of $z'(1|3\,2)$ is $S_{3}$-invariant and is isomorphic to the alternating representation of $S_{3}$. Since the $\Q$-span of $A'_{3,1}$ is trivially the $3^{\text{rd}}$-degree term of the FI$_{1}$-submodule of $H_{1}\big(\text{cell}(\bullet, 2);\Q\big)$ generated by the $\Q$-span of $A'_{3,1}$ and $[\iota_{1}]$, $A'_{3,1}$ satisfies property $3$.

The alternating representation of $S_{3}$ is isomorphic to $\bigwedge^{2}Std_{3}$, where $Std_{3}$ is the standard representation of $S_{3}$, and the $\Q$-span of $A'_{3,1}$ is trivially a subrepresentation of $3^{\text{rd}}$-degree term of the FI$_{1}$-submodule of $H_{1}\big(\text{cell}(\bullet, 2);\Q\big)$ generated by the $\Q$-span of $A'_{3,1}$ and $[\iota_{0}]$. Therefore, $A'_{3,1}$ satisfies property $4$.

Given $A'_{3,1}$, we construct sets $A'_{n,1}$ for all $n>3$ such that the $A'_{n,1}$ satisfy properties $1$ through $3$ of the theorem. To do this, first define a class of injected cycles. For $i_{1}, i_{2}, i_{3}\in [n]$ such that $i_{3}>i_{2}>i_{1}$, let $z'(i_{1}|i_{3}\,i_{2})$ be the injected cycle arising from $z'(1|3\,2)$ and the order preserving injection from $\{1,2,3\}$ into $\{i_{1}, i_{2},i_{3}\}$. For $n>3$, let $A'_{n,1}$ consist of all of the $1$-cycles of the form 
\[
z'(i_{1}|i_{3}\,i_{2}|i_{4}|\cdots|i_{n}):=z'(i_{1}|i_{3}\,i_{2})|z(i_{4})|\cdots|z(i_{n}),
\]
where $i_{j}\in [n]$, $i_{j}\neq i_{k}$ if $j\neq k$, $i_{3}>i_{2}>i_{1}$ and $i_{4}>\cdots>i_{n}$. The elements of $A_{n,1}$ are enumerated in the same way, so $|A'_{n,1}|=|A_{n,1}|$.

We have that
\[
z'(i_{1}|i_{3}\,i_{2})=z(i_{1}|i_{3}\,i_{2})+\frac{1}{2}\big(-z(i_{1}|i_{2}\,i_{3})+z(i_{3}|i_{1}\,i_{2})-z(i_{2}|i_{1}\,i_{3})+z(i_{2}\,i_{3}|i_{1})-z(i_{1}\,i_{2}|i_{3})+z(i_{1}\,i_{3}|i_{2})\big),
\]
so
\begin{multline*}
z'(i_{1}|i_{3}\,i_{2})|z(i_{4})|\cdots|z(i_{n})=\Big(z(i_{1}|i_{3}\,i_{2})+\frac{1}{2}\big(-z(i_{1}|i_{2}\,i_{3})+z(i_{3}|i_{1}\,i_{2})-z(i_{2}|i_{1}\,i_{3})\\
+z(i_{2}\,i_{3}|i_{1})-z(i_{1}\,i_{2}|i_{3})+z(i_{1}\,i_{3}|i_{2})\big)\Big)|z(i_{4})|\cdots|z(i_{n}).
\end{multline*}
This simplifies to
\begin{multline*}
z(i_{1}|i_{3}\,i_{2}|i_{4}|\cdots|i_{n})+\frac{1}{2}\Big(-z(i_{1}|i_{2}\,i_{3}|i_{4}|\cdots|i_{n})+z(i_{3}|i_{1}\,i_{2}|i_{4}|\cdots|i_{n})
-z(i_{2}|i_{1}\,i_{3}|i_{4}|\cdots|i_{n})\\
+z(i_{2}\,i_{3}|i_{1}|i_{4}|\cdots|i_{n})-z(i_{1}\,i_{2}|i_{3}|i_{4}|\cdots|i_{n})+z(i_{1}\,i_{3}|i_{2}|i_{4}|\cdots|i_{n})\Big).
\end{multline*}
Note that $z(i_{1}|i_{3}\,i_{2})$, $z(i_{3}|i_{1}\,i_{2})$, $z(i_{2}|i_{1}\,i_{3})$, $z(i_{3}\,i_{2}|i_{1})$, $z(i_{1}\,i_{2}|i_{3})$, and $z(i_{1}\,i_{3}|i_{2})$ are injected cycles arising from $B_{3}=B_{3,0}\cup B_{3,1}$. By Lemma \ref{switchsingles} we can reorder the singletons and get homotopic elements so that $z(i_{1}|i_{3}\,i_{2}|i_{4}|\cdots|i_{n})$, $z(i_{3}|i_{1}\,i_{2}|i_{4}|\cdots|i_{n})$, $z(i_{2}|i_{1}\,i_{3}|i_{4}|\cdots|i_{n})$, $z(i_{3}\,i_{2}|i_{1}|i_{4}|\cdots|i_{n})$, $z(i_{1}\,i_{2}|i_{3}|i_{4}|\cdots|i_{n})$, and $z(i_{1}\,i_{3}|i_{2}|i_{4}|\cdots|i_{n})$ are homotopic, and therefore homologous, to elements of $B_{n,0}\cup B_{n,1}$.

Note that $z'(i_{1}|i_{2}\,i_{3}|i_{4}|\cdots|i_{n})$ is the only element of $A'_{n,1}$ to contain $z(i_{1}|i_{2}\,i_{3}|i_{4}|\cdots|i_{n})\in A_{n,1}$ as a nontrivial summand, and each element of $A'_{n,1}$ contains only $1$ element of $A_{n,1}$ as a nontrivial summand. Since the basis elements of Theorem \ref{basistheorem}, and specifically the elements of $A_{n,1}$, are linearly independent, it follows that the elements of $A'_{n,1}$ are linearly independent, as the summands from $A_{n,1}$ cannot cancel out; hence, $A'_{n,1}$ satisfies property $1$. 

Since $z(i_{1}|i_{3}\,i_{2}|i_{4}|\cdots|i_{n})$, $z(i_{3}|i_{1}\,i_{2}|i_{4}|\cdots|i_{n})$, $z(i_{2}|i_{1}\,i_{3}|i_{4}|\cdots|i_{n})$, $z(i_{3}\,i_{2}|i_{1}|i_{4}|\cdots|i_{n})$, $z(i_{1}\,i_{2}|i_{3}|i_{4}|\cdots|i_{n})$, and $z(i_{1}\,i_{3}|i_{2}|i_{4}|\cdots|i_{n})$ are homologous to elements in $B_{n,1}\cup B_{n,0}$, the equations above and the previous paragraph prove that the $\Q$-span of $B_{n,0}\cup A'_{n,1}\cup B_{n,1}$ in $H_{1}\big(\text{cell}(n,2);\Q\big)$ is equal to the $\Q$-span of $B_{n,0}\cup A_{n,1}\cup B_{n,1}$ in $H_{1}\big(\text{cell}(n,2);\Q\big)$, i.e., $A'_{n,1}$ satisfies property $2$.

Next, we check that the sets $A'_{n,1}$ satisfy property $3$. We do this by showing that the elements of $A'_{n,1}$ are in the FI$_{1}$-submodule of $H_{1}\big(\text{cell}(\bullet, 2);\Q\big)$ generated by the $\Q$-span of $A'_{3,1}$ and the high-inclusion map $[\iota_{1}]$, that this FI$_{1}$-submodule is free, and then showing that the elements of $A'_{n, 1}$ span this free FI$_{1}$-module.

Note that 
\[
z'(1|3\,2|n|\cdots|4)=z'(1|3\,2)|z(n)|\cdots|z(4)=[\iota_{1}]^{n-3}z'(1|3\,2),
\]
where $[\iota_{1}]^{n-3}$ means we apply $[\iota_{1}]$ a total of $n-3$ times to $z'(1|3\,2)$, so $z'(1|3\,2|n|\cdots|4)$ is in the FI-submodule of $H_{1}\big(\text{cell}(\bullet, 2);\Q\big)$ generated by the $\Q$-span of $z'(1|3\,2)$ and $[\iota_{1}]$. Given $i_{1},\dots, i_{n}$ such that $i_{j}\in [n]$, $i_{j}\neq i_{k}$ if $j\neq k$, $i_{3}>i_{2}>i_{1}$ and $i_{n}>\cdots>i_{4}$, let $\sigma\in S_{n}$ be the permutation that sends $j$ to $i_{j}$. Then
\[
\sigma\big(z'(1|3\,2|n|\cdots|4)\big)=\sigma\big(z'(1|3\,2)\big)|\sigma\big(z(n)\big)|\cdots|\sigma\big(z(4)\big)=z'(i_{1}|i_{3}\,i_{2})|z(i_{n})|\cdots|z(i_{4})=z'(i_{1}|i_{3}\,i_{2}|i_{n}|\cdots|i_{4})
\]
Therefore, $z'(i_{1}|i_{3}\,i_{2}|i_{4}|\cdots|i_{n})$ is in the FI$_{1}$-submodule of $H_{1}\big(\text{cell}(\bullet, 2);\Q\big)$ generated by the $\Q$-span of $A'_{3,1}$ and $[\iota_{1}]$. Since every element in $A'_{n,1}$ is of this form, it follows that the $\Q$-span of $A'_{n,1}$ is in the FI$_{1}$-submodule of $H_{1}\big(\text{cell}(\bullet, 2);\Q\big)$ generated by the $\Q$-span of $A'_{3,1}$ and $[\iota_{1}]$.

Note that there are $\binom{n}{3}$ elements in $A'_{n,1}$. To see this, note that there are $\binom{n}{3}$ ways to choose $3$ elements from $[n]$ to be the labels $i_{1}, i_{2}, i_{3}$ such that $i_{1}<i_{2}<i_{3}$, and $1$ way to put the remaining $n-3$ in decreasing order. Since the elements of $A'_{n,1}$ are linearly independent, it follows that the $n^{\text{th}}$-degree term of the FI$_{1}$-submodule of $H_{1}\big(\text{cell}(\bullet, 2);\Q\big)$ generated by $A'_{3,1}$ and $[\iota_{1}]$ is at least $\binom{n}{3}$ dimensional.

Recall the $\Q$-span of $A'_{3,1}$ is isomorphic to $\bigwedge^{2}Std_{3}$ as an $S_{3}$-representation. The dimension of the $n^{\text{th}}$-degree of the free FI-module generated by $\bigwedge^{2}Std_{3}$ is 
\[
\dim\bigg(M^{\text{FI}}\Big(\bigwedge\nolimits^{2}Std_{3}\Big)_{n}\bigg)=\binom{n}{3}\dim\big(\bigwedge\nolimits^{2}Std_{3}\big)=\binom{n}{3}.
\]
Therefore, the $n^{\text{th}}$-degree term of the FI$_{1}$-submodule of $H_{1}\big(\text{cell}(\bullet, 2);\Q\big)$ generated by the $\Q$-span of $A'_{3,1}$ and $[\iota_{1}]$ is at most $\binom{n}{3}$ dimensional, as it is isomorphic to a FI$_{1}$-submodule of this free FI$_{1}$-module. Since the $\Q$-span of $A'_{n, 1}$ is in the $n^{\text{th}}$ degree of this FI$_{1}$-submodule and is $\binom{n}{3}$ dimensional it follows that $A'_{n,1}$ is basis for the $n^{\text{th}}$-degree of the FI$_{1}$-submodule of $H_{1}\big(\text{cell}(\bullet, 2);\Q\big)$ generated by the $\Q$-span of $A'_{3,1}$ and $[\iota_{1}]$. Since this is true for all $n$, the FI$_{1}$-submodule of $H_{1}\big(\text{cell}(\bullet, 2);\Q\big)$ generated by the $\Q$-span of $A'_{3,1}$ and $[\iota_{1}]$ is free and is isomorphic to $M^{FI}(\bigwedge^{2}Std_{3})$. Therefore $A'_{n,1}$ satisfies property $3$.


Now we assume that for $m\le k$ and $n\ge m+2$, there are sets $A'_{n,m}$ satisfying the theorem. We will show that for $m=k+1$, and $n\ge k+3$ there are sets $A'_{n,k+1}$ satisfying properties $1$ through $4$. We first show that there is a set $A'_{k+3,k+1}$ satisfying the properties of the theorem. We will use that set to build the sets $A'_{n,k+1}$ for $n>k+3$ much like we did for the sets $A'_{n,1}$.

By Corollary \ref{AlpertsBIGTheoremcor} and Corollary \ref{makesmallerFId} the $\Q$-span of $A'_{3,1}$ and $[\iota_{0}]$ generates an FI$_{1}$-submodule of $H_{1}\big(\text{cell}(\bullet, 2);\Q\big)$. Recall that $S_{3}$ acts as the alternating representation on the $\Q$-span of $A'_{3,1}$, which is isomorphic to $\bigwedge^{2}Std_{3}$; therefore, this FI$_{1}$-submodule is isomorphic to an FI$_{1}$-submodule of the free FI$_{1}$-module generated by $\bigwedge^{2}Std_{3}$. The $n^{\text{th}}$-degree term of the free FI$_{1}$-module generated by $\bigwedge^{2}Std_{3}$ decomposes into $2$ irreducible representations
\[
M^{\text{FI}}\Big(\bigwedge\nolimits^{2}Std_{3}\Big)_{n}=\bigwedge\nolimits^{2}Std_{n}\oplus \bigwedge\nolimits^{3}Std_{n}.
\]

Note,
\[
\dim\bigg(M^{\text{FI}}\Big(\bigwedge\nolimits^{2}Std_{3}\Big)_{n}\bigg) =\binom{n}{3},\indent \dim\Big(\bigwedge\nolimits^{2}Std_{n}\Big)=\binom{n-1}{2},\indent\text{and}\indent\dim\Big(\bigwedge\nolimits^{3}Std_{n}\Big)=\binom{n-1}{3}.
\]

Next, we calculate the dimension of $A_{k+3,k+1}$, as we must have that $|A'_{k+3,k+1}|=|A_{k+3,k+1}|$ if $A'_{k+3,k+1}$ is to satisfy property $1$. The set $A_{k+3,k+1}$ consists of all elements of the form $z(i_{k+3}|\cdots|i_{1}|i_{3}\,i_{2})$ where $i_{j}\in [k+3]$, $i_{j}\neq i_{l}$ for $j\neq l$, $i_{k+3}>\cdots>i_{4}>i_{1},$ and $i_{3}>i_{2}>i_{1}$. By Theorem \ref{basistheorem} the elements of $A_{k+3,k+1}$ are linearly independent. The conditions $i_{j}\in [k+3]$, $i_{j}\neq i_{l}$ for $j\neq l$, $i_{k+3}>\cdots>i_{4}>i_{1},$ and $i_{3}>i_{2}>i_{1}$ imply that $i_{1}=1$. There are $\binom{k+2}{2}$ ways to choose $i_{2}$ and $i_{3}$ from the remaining $k+2$ elements $2,\dots, k+3$ such that $i_{3}>i_{2}$, and there is one way to put the remaining $k$ elements in decreasing order. Thus $|A_{k+3,k+1}|=\binom{k+2}{2}$. Thus, for a set $A'_{k+3,k+1}$ to satisfy property $1$ of the theorem we must have that $|A'_{k+3,k+1}|=\binom{k+2}{2}$, and the elements of $A'_{k+3,k+1}$ must be linearly independent. We claim that the $\bigwedge^{2}Std_{k+3}$ part of the $(k+3)^{\text{th}}$-degree of the FI$_{1}$-submodule of $H_{1}\big(\text{cell}(\bullet, 2);\Q)\big)$ generated by the $\Q$-span of $A'_{3,1}$ and $[\iota_{0}]$ is nontrivial in $H_{1}\big(\text{cell}(k+3,2);\Q\big)$, and that any basis for this representation will work as our set $A'_{k+3,k+1}$.

First, we show that the $(k+3)^{\text{th}}$-degree term of the FI$_{1}$-submodule of $H_{1}\big(\text{cell}(\bullet,2);\Q\big)$ generated by the $\Q$-span of $A'_{3,1}$ and $[\iota_{0}]$ is nontrivial in $H_{1}\big(\text{cell}(k+3,2);\Q\big)$. To see this note that 
\[
[\iota_{0}]^{k}z'(1|3\,2)=z(k+3)|\cdots|z(4)|z'(1|3\,2).
\]
Since $z'(1|3\,2)$ is a nontrivial linear combination of Alpert's basis elements it follows that this concatenation product is a nontrivial cycle in $H_{1}\big(\text{cell}(k+3,2);\Q\big)$, and the $(k+3)^{\text{th}}$-degree term of this FI-module is nontrivial.

Every cycle of the form
\[
z(i_{k})|\cdots|z(i_{1})|z'(i_{k+1}|i_{k+3}\,i_{k+2}),
\]
where $i_{j}\in [k+3]$, $i_{j}\neq i_{l}$ for $j\neq l$, $i_{k}>\cdots>i_{1}>i_{k+1}$, and $i_{k+3}>i_{k+2}>i_{k+1}$ is nonzero in this FI$_{1}$-submodule. This can be seen by applying the symmetric group element $\sigma$ that sends $j\mapsto i_{j-3}$ for $4\le j\le k+3$, and $1\mapsto i_{k+1}$, $2\mapsto i_{k+2}$, and $3\mapsto i_{k+3}$ to $z(k+3)|\cdots|z(4)|z'(1|3\,2)$. We can decompose these elements as sums of basis elements from Theorem \ref{basistheorem}:
\begin{multline*}
z(i_{k})|\cdots|z(i_{1})|z'(i_{k+1}|i_{k+3}\,i_{k+2})=z(i_{1})|\cdots|z(i_{k})|\Big(z(i_{k+1}|i_{k+3}\,i_{k+2})+\frac{1}{2}\big(-z(i_{k+1}|i_{k+2}\,i_{k+3})\\
+z(i_{k+2}|i_{k+1}\,i_{k+3})-z(i_{k+3}|i_{k+1}\,i_{k+2})+z(i_{k+2}\,i_{k+3}|i_{k+1})-z(i_{k+1}\,i_{k+3}|i_{k+2})+z(i_{k+1}\,i_{k+2}|i_{k+3})\big)\Big),
\end{multline*}
which is equal to
\begin{multline*}
z(i_{k}|\cdots|i_{1}|i_{k+1}|i_{k+3}\,i_{k+2})+\frac{1}{2}\big(-z(i_{k}|\cdots|i_{1}|i_{k+1}|i_{k+2}\,i_{k+3})+z(i_{k}|\cdots|i_{1}|i_{k+2}|i_{k+1}\,i_{k+3})\\-z(i_{k}|\cdots|i_{1}|i_{k+3}|i_{k+1}\,i_{k+2})
+z(i_{k}|\cdots|i_{1}|i_{k+2}\,i_{k+3}|i_{k+1})\\-z(i_{k}|\cdots|i_{1}|i_{k+1}\,i_{k+3}|i_{k+2})+z(i_{k}|\cdots|i_{1}|i_{k+1}\,i_{k+2}|i_{k+3})\big).
\end{multline*}
Note that this decomposition contains $z(i_{k}|\cdots|i_{1}|i_{k+1}|i_{k+3}\,i_{k+2})$, where $i_{j}\in [k+3]$, $i_{j}\neq i_{l}$ for $j\neq l$, $i_{k}>\cdots>i_{1}>i_{k+1}$, and $i_{k+3}>i_{k+2}>i_{k+1}$ as a nontrivial summand. These are the elements of  $A_{k+3,k+1}$, and every element in $A_{k+3,k+1}$ appears in exactly one of these decompositions, as the other elements in these decompositions are homotopic to elements in $B_{k+3,k+1}\cup B_{k+3,k}$ by Lemma \ref{switchsingles}.

By the induction hypothesis the $\Q$-span of $B_{k+3,0}\bigcup_{1\le l\le k}(A'_{k+3,l}\cup B_{k+3,l})$ in $H_{1}\big(\text{cell}(k+3, 2);\Q\big)$ is equal to the $\Q$-span of $B_{k+3,0}\bigcup_{1\le l\le k}(A_{k+3,l}\cup B_{k+3,l})$ in $H_{1}\big(\text{cell}(k+3, 2);\Q\big)$. Thus, for every element $z$ in $\Q$-span of $B_{k+3,0}\bigcup_{1\le l\le k}(A'_{k+3,l}\cup B_{k+3,l})$, the leading value of $z$, is at most $2k$, i.e., $L(z)\le 2k$. Note that any element in the $\Q$-span of $B_{k+3,k+1}$ has leading value $2k+1$, so the $\Q$-span of $B_{k+3,k+1}$ in $H_{1}\big(\text{cell}(k+3, 2);\Q\big)$, and the $\Q$-span of $B_{k+3,0}\bigcup_{1\le l\le k}(A'_{k+3,l}\cup B_{k+3,l})$ in $H_{1}\big(\text{cell}(k+3, 2);\Q\big)$, which is the same as the $\Q$-span of $B_{k+3,0}\bigcup_{1\le l\le k}(A_{k+3,l}\cup B_{k+3,l})$ in $H_{1}\big(\text{cell}(k+3, 2);\Q\big)$, must be disjoint. This, along with the previous paragraph, show that the $\Q$-span of the union of the $(k+3)^{\text{th}}$-degree term of the FI$_{1}$-submodule of $H_{1}\big(\text{cell}(\bullet, 2);\Q\big)$ generated by $A'_{3,1}$ and $[\iota_{0}]$ and $B_{k+3,0}\cup B_{k+3,k+1}\bigcup_{1\le l\le k}(A'_{k+3,l}\cup B_{k+3,l})$ in $H_{1}\big(\text{cell}(k+3, 2);\Q\big)$ is equal to the $\Q$-span $B_{k+3,0}\bigcup_{1\le l\le k+1}(A_{k+3,l}\cup B_{k+3,l})$ in $H_{1}\big(\text{cell}(k+3, 2);\Q\big)$, i.e., $H_{1}\big(\text{cell}(k+3,2);\Q\big)$.

By the induction hypothesis the $\Q$-span of $B_{k+3,0}\cup B_{k+3,k+1}\bigcup_{1\le l\le k}A'_{k+3,l}\cup B_{k+3,l}$ is $S_{k+3}$-invariant in $H_{1}\big(\text{cell}(k+3, 2);\Q\big)$, as $B_{k+3,0}\cup B_{k+3,k+1}\bigcup_{1\le l\le k}B_{k+3,l}=B_{k+3}$, and $B_{k+3}$ and the $A'_{k+3,l'}$s span $S_{k+3}$-subrepresentations of $H_{1}\big(\text{cell}(k+3, 2);\Q\big)$. Since the $\Q$-span of the $(k+3)^{\text{rd}}$-degree term of the FI$_{1}$-submodule of $H_{1}\big(\text{cell}(\bullet, 2);\Q\big)$ generated by $A'_{3,1}$ and $[\iota_{0}]$ together with the $\Q$-span of $B_{k+3,0}\cup B_{k+3,k+1}\bigcup_{1\le l\le k}(A'_{k+3,l}\cup B_{k+3,l})$ in $H_{1}\big(\text{cell}(k+3, 2);\Q\big)$ is $H_{1}(\text{cell}(k+3,2);\Q\big)$, it follows that the inclusion of $(k+3)^{\text{rd}}$-degree term of the FI$_{1}$-submodule of $H_{1}\big(\text{cell}(\bullet, 2);\Q\big)$ generated by $A'_{3,1}$ and $[\iota_{0}]$ surjects onto the quotient of $H_{1}\big(\text{cell}(k+3,2);\Q\big)$ by the $\Q$-span of $B_{k+3,0}\cup B_{k+3,k+1}\cup_{1\le l\le k}(A'_{k+3,l}\cup B_{k+3,l})$, as $H_{1}\big(\text{cell}(k+3,2);\Q\big)$ is a semi-simple. Note that the dimension of this quotient space is the dimension of the span of $A_{k+3,k+1}$, i.e., $\binom{k+2}{3}$. 

When $k\neq 3$, it follows that $\dim(\bigwedge^{3}Std_{k+3})=\binom{k+2}{3}\neq\binom{k+2}{2}$. 
Since $\dim(\bigwedge^{2}Std_{k+3})=\binom{k+2}{2}$, and the FI$_{1}$-submodule of $H_{1}\big(\text{cell}(\bullet,2);\Q\big)$ generated by $A'_{3,1}$ and $[\iota_{0}]$ is isomorphic to a nontrivial summand of the free FI$_{1}$-module generated by $\bigwedge^{2}Std_{3}$, it follows from the Schur Lemma that the $\bigwedge^{3}Std_{k+3}$ part of the FI$_{1}$-submodule of $H_{1}\big(\text{cell}(\bullet,2);\Q\big)$ generated by $A'_{3,1}$ and $[\iota_{0}]$ is trivial under the map to the quotient space. 
Since the map from the FI$_{1}$-submodule of $H_{1}\big(\text{cell}(\bullet,2);\Q\big)$ generated by $A'_{3,1}$ and $[\iota_{0}]$ to the quotient is surjective, it follows that the $\bigwedge^{2}Std_{k+3}$ part of the FI$_{1}$-submodule of $H_{1}\big(\text{cell}(\bullet,2);\Q\big)$ generated by $A'_{3,1}$ and $[\iota_{0}]$ is isomorphic to the image of this map. 
Thus, the $\bigwedge^{2}Std_{k+3}$ part of the FI$_{1}$-module generated by $A'_{3,1}$ and $[\iota_{0}]$ direct sum the $\Q$-span of $B_{k+3,0}\cup B_{k+3,k+1}\bigcup_{1\le l\le k}(A'_{k+3,l}\cup B_{k+3,l})$ is equal to the $\Q$-span $B_{k+3,0}\cup\bigcup_{1\le l\le k+1}(A_{k+3,l}\cup B_{k+3,l})$, i.e., $H_{1}\big(\text{cell}(k+3,2);\Q\big)$. 
Any basis for the $\bigwedge^{2}Std_{k+3}$ part of the FI$_{1}$-submodule of $H_{1}\big(\text{cell}(\bullet,2);\Q\big)$ generated by $A'_{3,1}$ and $[\iota_{0}]$ works as our set $A'_{k+3,k+1}$ as it clearly satisfies point $1$, we have just shown that our sets $A'_{k+3, m}$ for $m\le k+1$ and $k\neq 3$ all satisfy points $2$ and $4$, and they trivially satisfies point $3$. 
Thus, all we have to do is check that this is true for when $k=3$.

In the case $k=3$ we need to show that there is a set $A'_{6,4}$ with the desired properties. By the induction hypothesis there exist sets $A'_{n,m}$ satisfying the properties of the theorem for all $m\le 3$ and $n\ge m+2$. As noted above, our argument for $n=6, m=4$ is slightly different as we cannot distinguish the irreducible representations $\bigwedge^{2}Std_{6}$ and $\bigwedge^{3}Std_{6}$ by counting dimensions as 
\[
\dim\Big(\bigwedge\nolimits^{2}Std_{6}\Big)=\binom{5}{2}=10=\binom{5}{3}=\dim\Big(\bigwedge\nolimits^{3}Std_{6}\Big).
\]
Instead, we use character theory to determine which of these two representations has yet to be accounted for in our decomposition of $H_{1}\big(\text{cell}(6,2);\Q\big)$.

The character of the transposition $\tau=(a b)\in S_{6}$ for the irreducible representation $\bigwedge^{2}Std_{6}$ is $2$ and for the irreducible representation $\bigwedge^{3}Std_{6}$ the character of $\tau$ is $-2$, see, for example, the character table for $S_{6}$ in \cite{james2001representations}. Since the $\Q$-span of $B_{6,4}\cup B_{6,0}\bigcup_{1\le l\le 3}(A'_{6,l}\cup B_{6,l})$ in $H_{1}\big(\text{cell}(6, 2);\Q\big)$ is equal to the $\Q$-span of $B_{6,4}\cup B_{6,0}\bigcup_{1\le l\le 3}(A_{6,l}\cup B_{6,l})$ in $H_{1}\big(\text{cell}(6, 2);\Q\big)$, it follows that we just need to determine the trace of $\tau$ on $A_{6,4}$ to decide which representation should be equal to the $\Q$-span of $A'_{6,4}$.

To determine the trace of the conjugacy class of the simple transpositions $(ab)\in S_{6}$ it suffices to determine the action of any transposition on a basis for $H_{1}\big(\text{cell}(6,2);\Q\big)$, and by the above it suffices to determine the action of any transposition on the homology classes represented by elements of $A_{6,4}$. For our purposes the transposition $(23)$ will yield the easiest calculations. 

Recall that the $10$ elements of $A_{6,4}$ are $z(6|5|4|1|3\,2)$, $z(6|3|2|1|5\,4)$, $z(5|3|2|1|6\,4)$, $z(4|3|2|1|6\:5)$, $z(6|5|3|1|4\,2)$, $z(6|5|2|1|4\,3)$, $z(6|4|3|1|5\,2)$, $z(6|4|2|1|5\,3)$, $z(5|4|3|1|6\,2)$, and $z(5|4|2|1|6\,3)$. 

We have that
\begin{align*}
(23)z(6|5|4|1|3\,2)=(23)z(6|5|4)|z(1|3\,2)=z(6|5|4)|\big((23)z(1|3\,2)\big)\\
=z(6|5|4)|\big((23)(1|3\,2+3\,1|2+3|2\,1-3\,2|1-2|3\,1-2\,1|3)\big)\\
=z(6|5|4)|\big(1|2\,3+2\,1|3+2|3\,1-2\,3|1-3|2\,1-3\,1|2)\big)\\
=z(6|5|4)|\big(-z(1|3\,2)+z(1|2\,3)-z(2\,3|1)\big)\\
=-z(6|5|4)|z(1|3\,2)+z(6|5|4)|z(1|2\,3)-z(6|5|4)|z(2\,3|1)\\
=-z(6|5|4|1|3\,2)+z(6|5|4|1|2\,3)-z(6|5|4|2\,3|1)
\end{align*}
Thus, the action of $(23)$ on $z(6|5|4|1|3\,2)$ adds a $-1$ to the trace.

Next note that,
\begin{align*}
(23)z(6|3|2|1|5\,4)=(23)\big(z(6)|z(3|2)|z(1|5\,4)\big)=z(6)|\big((23)z(3)|z(2)\big)|z(1|5\,4)\\
=z(6)|z(2)|z(3)|z(1|5\,4)=z(6)|z(3)|z(2)|z(1|5\,4),
\end{align*}
where we have used Proposition \ref{switchsingles} to write $z(2)|z(3)$ as $z(3)|z(2)$. This is homologous to
\[
z(6|3|2|1|5\,4).
\]
Thus, the action of $(23)$ on $z(6|3|2|1|5\,4)$ adds a $1$ to the trace. Similarly we can see that the action of $(23)$ on $z(5|3|2|1|6\,4)$ and $z(4|3|2|1|6\,5)$ will each add $1$ to the trace.

Next, we determine the action of $(23)$ on $z(6|5|3|1|4\,2)$ and $z(6|5|2|1|4\,3)$.
\begin{align*}
(23)z(6|5|3|1|4\,2)=(23)\big(z(6|5)|z(3)|z(1|4\,2)\big)=z(6|5)|\big((23)z(3)|z(1|4\,2)\big)\\
=z(6|5)|\big(z(2)|z(1|4\,3)\big)=z(6|5)|z(2)|z(1|4\,3)=z(6|5|2|1|4\,3).
\end{align*}
This also shows that $(23)z(6|5|2|1|4\,3)=z(6|5|3|1|4\,2)$. Thus the action of $\tau$ on $z(6|5|3|1|4\,2)$ and $z(6|5|2|1|4\,3)$ adds $0$ to the trace. We can check that $(23)$ similarly interchanges the homology classes of represented by the basis elements $z(6|4|3|1|5\,2)$ and $z(6|4|2|1|5\,3)$, and it also interchanges the homology classes represented by the basis elements $z(5|4|3|1|6\,2)$ and $z(5|4|2|1|6\,3)$, so these elements contribute $0$ to the trace of the action of $(23)$ on $H_{1}\big(\text{cell}(6,2);\Q\big)$.

Thus, we see that the action of $(23)$ on the elements of $A_{6,4}$ contributes $-1+1+1+1+0+0+0+0+0+0=2$ to the trace. Therefore, the character of a transposition on the remaining $10$-dimensional representation is $2$. Therefore, $H_{1}\big(\text{cell}(6,2);\Q\big)$ decomposes as a direct sum of the space spanned by the homology classes represented by elements of $B_{6,4}\cup B_{6,0}\bigcup_{1\le l\le 3}(A'_{6,l}\cup B_{6,l})$ and a space that is isomorphic to $\bigwedge^{2}Std_{6}$ as $S_{6}$-representations. Any basis for this space isomorphic to $\bigwedge^{2}Std_{6}$ in the $6^{\text{th}}$-degree term of the FI$_{1}$-submodule of $H_{1}\big(\text{cell}(\bullet, 2);\Q\big)$ generated by $z'(1|3\,2)$ and $[\iota_{0}]$ will work as our set $A'_{6,4}$, and will satisfy the properties of the theorem. 

Now that we have that there exists a set $A'_{k+3, k+1}$ satisfying the properties of the theorem in conjunction with sets $A'_{n,m}$ for $m\le k$ that satisfy the properties of theorem, we show that we can construct sets $A'_{n, k+1}$ for all $n\ge k+3$ satisfying the theorem. 

First, note that since we've found a set $A'_{k+3,k+1}$ satisfying the properties of the theorem, namely, property $1$, we can index elements of $A'_{k+3, k+1}$ in the same way as we do elements of $A_{k+3,k+1}$. As such, we denote the elements of $A'_{k+3,k+1}$ by
\[
z'(i_{1}|\dots|i_{k+1}|i_{k+3}\,i_{k+2}),
\]
where $i_{j}\in \{1,\dots, k+3\}$, $i_{j}\neq i_{l}$ for $j\neq l$, $i_{1}>\cdots>i_{k+1}$, and $i_{k+3}>i_{k+2}>i_{k+1}$.

For $n\ge k+3$ let $A'_{n,k+1}$ consist of cycles of the form
\[
z'(i_{1}|\cdots|i_{k+1}|i_{k+2}\,i_{k+3}|i_{4}|\cdots|i_{k+4}):=z'(i_{1}|\cdots|i_{k+1}|i_{k+2}\,i_{k+3})|z(i_{n})|\cdots|z(i_{k+4}),
\]
where $i_{j}\in [n]$, $i_{j}\neq i_{l}$ for $j\neq l$, $i_{1}>\cdots>i_{k+1}$, $i_{k+2}>i_{k+3}>i_{k+1}$, so that $z'(i_{1}|\cdots|i_{k+1}|i_{k+2}\,i_{k+3})$ is a properly ordered injected cycle arising from $A'_{k+3,k+1}$, and $i_{n}>\cdots>i_{k+4}$. We need to show that these sets $A'_{n,k+1}$ have properties $1$ through $3$.

To see that the elements of $A'_{n,k+1}$ are linearly independent note that by the induction hypothesis the $\Q$-span of $B_{k+3,0}\bigcup_{1\le l\le k+1}(A'_{k+3,l}\cup B_{k+3,l})$ in $H_{1}\big(\text{cell}(k+3, 2);\Q\big)$ equals the $\Q$-span of $B_{k+3,0}\bigcup_{1\le l\le k+1}(A_{k+3,l}\cup B_{k+3,l})$ in $H_{1}\big(\text{cell}(k+3, 2);\Q\big)$, so we can write every element of $A_{k+3,k+1}$ as a linear combination of terms in $A'_{k+3,k+1}$ and elements in 
\[
B_{k+3,0}\cup B_{k+3,k+1}\bigcup_{1\le l\le k}\left(A_{k+3,l}\cup B_{k+3,l}\right).
\] 
Thus, for each element $z(i_{1}|\cdots|i_{k+1}|i_{k+2}\,i_{k+3})\in A_{k+3,k+1}$ there are sums $\sum a_{e}z'(e)$ and $\sum a_{e'}z(e')$, such that $z'(e)\in A'_{k+3,k+1}$ and $z(e')\in B_{k+3,0}\cup B_{k+3,k+1}\bigcup_{1\le l\le k}\left(A_{k+3,l}\cup B_{k+3,l}\right)$ such that
\[
\sum a_{e}z'(e)+\sum a_{e'}z(e')=z(i_{1}|\cdots|i_{k+1}|i_{k+2}\,i_{k+3}).
\]
Thus, for every element
 \[
z(i_{1}|\cdots|i_{k+1}|i_{k+2}\,i_{k+3})|z(i_{n})|\cdots|z(i_{k+4})=z(i_{1}|\cdots|i_{k+1}|i_{k+2}\,i_{k+3}|i_{n}|\cdots|i_{k+4})\in A_{n,k+1},
\]
there is some linear combination
\[
\sum a_{e}z'(e)|z(i_{n})|\cdots|z(i_{k+4})+\sum a_{e'}z(e')|z(i_{n})|\cdots|z(i_{k+4})=z(i_{1}|\cdots|i_{k+1}|i_{k+2}\,i_{k+3})|z(i_{n})|\cdots|z(i_{k+4}),
\]
where $z'(e)|z(i_{k+3})|\cdots|z(i_{n})\in A'_{n,k+1}$, and $z(e')|z(i_{k+3})|\cdots|z(i_{n})\in B_{n,0}\cup B_{n,k+1}\bigcup_{1\le l\le k}\left(A_{n,l}\cup B_{n,l}\right)$. Therefore, the $\Q$-span of $A'_{n,k+1}\cup B_{n,0}\cup B_{n,k+1}\bigcup_{1\le l\le k}(A_{n,l}\cup B_{n,l})$  in $H_{1}\big(\text{cell}(n, 2);\Q\big)$ is equal to the $\Q$-span of $A_{n,k+1}\cup B_{n,0}\cup B_{n,k+1}\bigcup_{1\le l\le k}(A_{n,l}\cup B_{n,l})$ in $H_{1}\big(\text{cell}(n, 2);\Q\big)$. Since the basis elements of Theorem \ref{basistheorem} are linearly independent it follows that there is an $|A_{n,k+1}|$ dimensional subspace in the $\Q$-span of $A'_{n,k+1}$. Since the elements of $A_{n,k+1}$ are of the form $z'(i_{1}|\cdots|i_{k+1}|i_{k+2}\,i_{k+3})|z(i_{n})|\cdots|z(i_{k+4})$ where $i_{j}\in [n]$, $i_{j}\neq i_{l}$ for $j\neq l$, $i_{1}>\cdots>i_{k+1}$, $i_{k+2}>i_{k+3}>i_{k+1}$, and $i_{n}>\cdots>i_{k+4}$, the elements of $A'_{n,k+1}$ and $A_{n,k+1}$ are indexed in the same way. Thus $|A'_{n,k+1}|=|A_{n,k+1}|$. Since the $\Q$-span of $A'_{n,k+1}$ has an $|A_{n,k+1}|$-dimensional subspace the elements of $A'_{n,k+1}$ are linearly independent, i.e., $A'_{n,k+1}$ has property $1$. This argument also shows that our set $A'_{n,k+1}$ has property $2$. 

It remains to show that the set $A'_{n,k+1}$ has property $3$. Given a set of $k+3$ distinct elements $\{i_{1}, \dots, i_{k+3}\}$ in $[n]$ let $z'(i_{1}|\cdots|i_{k+1}|i_{k+2}\,i_{k+3})$ denote the injected cycle arising from $z'(j_{1}|\cdots|j_{k+1}|j_{k+2}\,j_{k+3})$ where $j_{l}\in [k+3]$, $j_{s}\neq j_{r}$ if $s\neq r$, $j_{1}>\cdots>j_{k+1}$, and $j_{k+2}>j_{k+3}>j_{k+1}$, where the map is order preserving, i.e., if $j_{r}<j_{s}$ the $i_{r}<i_{s}$. Let $\sigma\in S_{n}$ be the permutation that sends $j_{r}$ to $i_{r}$ and $t$ to $i_{t}$ for $t\ge k+4$. Then,
\[
\sigma(z'(j_{1}|\cdots|j_{k+1}|j_{k+2}\,j_{k+3})|z(n)|\cdots|z(k+4))=z'(i_{1}|\cdots|i_{k+1}|i_{k+2}\,i_{k+3})|z(i_{n})|\cdots|z(i_{k+4}).
\]
Note that $z'(j_{1}|\cdots|j_{k+1}|j_{k+2}\,j_{k+3})|z(n)|\cdots|z(k+4)=[\iota_{1}]^{n-k-3}z'(j_{1}|\cdots|j_{k+1}|j_{k+2}\,j_{k+3})$, it follows that the elements of our set $A'_{n,k+1}$ are contained in the FI$_{1}$-submodule of $H_{1}\big(\text{cell}(\bullet, 2);\Q\big)$ generated by $A'_{k+3, k+1}$ and $[\iota_{1}]$. Recall that the $\Q$-span of $A'_{k+3,k+1}$ is isomorphic to $\bigwedge^{2}Std_{k+3}$ as a $S_{k+3}$-representation. Note that the $n^{\text{th}}$-degree part of the free FI$_{1}$-module generated by $\bigwedge^{2}Std_{k+3}$ has dimension
\[
\dim\bigg(M^{\text{FI}_{1}}\Big(\bigwedge\nolimits^{2}Std_{k+3}\Big)_{n}\bigg)=\dim\Big(\bigwedge\nolimits^{2}Std_{k+3}\Big)\binom{n}{k+3}=\binom{k+2}{2}\binom{n}{k+3}.
\]
There are $\binom{k+2}{2}\binom{n}{k+3}$ elements of the form $z'(i_{1}|\cdots|i_{k+1}|i_{k+2}\,i_{k+3})|z(i_{n})|\cdots|z(i_{k+4}),$ where $i_{j}\in [n]$, $i_{j}\neq i_{l}$ for $j\neq l$, $i_{1}>\cdots>i_{k+1}$,$i_{k+2}>i_{k+3}>i_{k+1}$, and $i_{n}>\cdots>i_{k+4}$, as there are $\binom{n}{k+3}$ ways to choose $i_{1},\dots i_{k+3}$, and $\binom{k+2}{2}$ ways to put them in order so that they index an element of $A'_{k+3,k+1}$. There is $1$ way to put the remaining elements in decreasing order. Thus, the $\Q$-span of $A'_{n, k+1}$ in $H_{1}\big(\text{cell}(n, 2);\Q\big)$ is a subset of the FI$_{1}$-submodule of $H_{1}\big(\text{cell}(\bullet, 2);\Q\big)$ generated by $A'_{k+3,k+1}$ and $[\iota_{1}]$ is isomorphic to the free FI$_{1}$-module $M^{\text{FI}}\Big(\bigwedge\nolimits^{2}Std_{k+3}\Big)$, and is therefore free. This hold for all $n$, so $A'_{n,k+1}$ satisfies property $3$. This completes the induction step, and we have shown that for $n\ge 3$ and $1\le m\le n-2$ there are sets $A'_{n,m}$ with the properties of the theorem, proving the theorem.
\end{proof}

Next, we show that given a family of sets $\{A'_{n,m}\}_{1\le m\le n-2}$ satisfying the properties of Theorem \ref{betterbasisforH1}, we have that $B_{n,0}\bigcup_{1\le m\le n-2}(A'_{n,m}\cup B_{n,m})$ is a basis for $H_{1}\big(\text{cell}(n,2);\Q\big)$.

\begin{cor}\label{isabasisforH1}
Let for $1\le m\le n-2$ let  $A'_{n,m}$ satisfy the properties of Theorem \ref{betterbasisforH1}. Set $A'_{n}:=\bigcup_{m=1}^{n-2}A'_{n,m}$. Then $A'_{n}\cup B_{n}$ is a basis for $H_{1}\big(\text{cell}(n,2);\Q\big)$.
\end{cor}

\begin{proof}
Since our sets $A'_{n,m}$ have satisfies property $1$ of Theorem \ref{betterbasisforH1} we have $|A'_{n,m}|=|A_{n,m}|$. Therefore, the cardinality of $B_{n,0}\bigcup_{1\le m\le n-2}(A'_{n,m}\cup B_{n,m})$ is at most the cardinality of $B_{n,0}\bigcup_{1\le m\le n-2}(A_{n,m}\cup B_{n,m})$. Since the sets $A'_{n,m}$ satisfy property $2$ of Theorem \ref{betterbasisforH1} the $\Q$-span of $B_{n,0}\bigcup_{1\le m\le n-2}(A'_{n,m}\cup B_{n,m})$ in $H_{1}\big(\text{cell}(n, 2);\Q\big)$ equals the $\Q$-span of $B_{n,0}\bigcup_{1\le m\le n-2}(A_{n,m}\cup B_{n,m})$ in $H_{1}\big(\text{cell}(n, 2);\Q\big)$, so they must have the same cardinality. By Theorem \ref{basistheorem} $B_{n,0}\bigcup_{1\le m\le n-2}(A_{n,m}\cup B_{n,m})$ is a basis for $H_{1}\big(\text{cell}(n,2);\Q\big)$, so $B_{n,0}\bigcup_{1\le m\le n-2}(A'_{n,m}\cup B_{n,m})=A'_{n}\cup B_{n}$ is a basis for $H_{1}\big(\text{cell}(n,2);\Q\big)$.
\end{proof}

\begin{note}
The elements of the basis for $H_{1}\big(\text{cell}(n,2);\Q\big)$ of Theorem \ref{basistheorem} can be realized by embedded circles. Unfortunately, this basis is not well behaved with respect to the symmetric group action. With Theorem \ref{betterbasisforH1} and Corollary \ref{isabasisforH1} we have found a new basis on which the symmetric group action is far simpler, though this comes at the cost of topological complexity. Our new basis retains the simplest $1$-cycles from \ref{basistheorem}, i.e., $z(1\,2)$, but has $1$-cycles arising from the movement of arbitrarily many disks, so that the elements in our new basis must be realized as rational sums of arbitrarily many embedded circles.
\end{note}

\begin{thm}\label{repstructofH1}
As an $S_{n}$-representation $H_{1}\big(\text{cell}(n,2);\Q\big)$ is isomorphic to 
\[
M^{\text{FI}_{2}}(Triv_{2})_{n}\bigoplus_{1\le m\le n-2}M^{\text{FI}}\Big(\bigwedge\nolimits^{2}Std_{m+2}\Big)_{n}.
\]
\end{thm}

\begin{proof}
Fix a set of $A'_{n,m}$ satisfying the properties of Theorem \ref{betterbasisforH1}. By Prop \ref{BnisabasisforthefreeFI2mod} the $\Q$-span of $B_{n}$ is isomorphic to $M^{\text{FI}_{2}}(Triv_{2})_{n}$ as an $S_{n}$-representation. By Theorem \ref{betterbasisforH1} the $\Q$-span of $A'_{n,m}$ is isomorphic to $M^{\text{FI}}(\bigwedge^{2}Std_{m+2})_{n}$ as an $S_{n}$-representation. By Corollary \ref{isabasisforH1} $A'_{n}\cup B_{n}$ is a basis for $H_{1}(\text{cell}(n,2);\Q)$, so as an $S_{n}$-representation $H_{1}\big(\text{cell}(n,2);\Q\big)$ is isomorphic to
\[
M^{\text{FI}_{2}}(Triv_{2})_{n}\bigoplus_{1\le m\le n-2}M^{\text{FI}}\Big(\bigwedge\nolimits^{2}Std_{m+2}\Big)_{n}.
\]
\end{proof}

Thus, we have found a basis that allows us to determine $S_{n}$-representation structure $H_{1}\big(\text{cell}(n,2);\Q\big)$. We will use these basis elements to construct free FI$_{d}$-submodules of $H_{k}\big(\text{cell}(n,2);\Q\big)$. Before we do that, we return to idea of leading value as it will be useful in the following section.

The leading value $L$ gives an ordering on the sets $A_{m,n}$ and $B_{n,m}$, and we can think of this ordering as saying that the cycles in $A_{n,m}$ (and as we will soon see $A'_{n,m}$) are more complex than cycles in $A_{n,m'}$ (respectively, $A'_{n,m'}$) if $m>m'$. In the case of $A'_{n,m}$ we can think of this complexity as needing $m+2$ moving disks to create a cycle.

\begin{prop}\label{Lfunction}
For every $z(e)\in A_{n,m}$ there is a unique linear combination of $z'(e_{i})\in A'_{n,m}$ such that
\[
\sum a_{i}z'(e_{i})=z(e)+\sum a_{e'}z'(e')
\]
where
\[
L\Big(\sum a_{e'}z'(e')\Big)<2m.
\]
\end{prop}

\begin{proof}
By Theorem \ref{betterbasisforH1} the $\Q$-span of $B_{n,0}\bigcup_{1\le l\le m}(A'_{n, l}\cup B_{n,l})$ in $H_{1}\big(\text{cell}(n,2);\Q\big)$ is equal to the $\Q$-span of $B_{n,0}\bigcup_{1\le l\le m}(A_{n, l}\cup B_{n,l})$ in $H_{1}\big(\text{cell}(n,2);\Q\big)$. By Theorem \ref{basistheorem} the elements of these sets are linearly independent. Thus, we can write every element of $A_{n,m}$ as a unique linear combination of elements in $A'_{n,m}$ and elements in $B_{n,m}\cup B_{n,0}\bigcup_{1\le l\le m-1}(A'_{n, l}\cup B_{n,l})$. Therefore, 
\[
\sum a_{i}z'(e_{i})=z(e)+\sum a_{e'}z'(e')
\]
where $z'(e')\in B_{n,m}\cup B_{n,0}\bigcup_{1\le l\le m-1}(A'_{n, l}\cup B_{n,l})$. We have that $L\big(z'(e)\big)<2m$ for all $z'(e')$ as Theorem \ref{betterbasisforH1} proves that the $\Q$-span of $B_{n,m}\cup B_{n,0}\bigcup_{1\le l\le m-1}(A'_{n, l}\cup B_{n,l})$ in $H_{1}\big(\text{cell}(n,2);\Q\big)$  is equal to the $\Q$-span of $B_{n,m}\cup B_{n,0}\bigcup_{1\le l\le m-1}(A_{n, l}\cup B_{n,l})$ in $H_{1}\big(\text{cell}(n,2);\Q\big)$, and elements in this sub-basis have leading value less than $2m$.
\end{proof}

\begin{prop}\label{Lfunction2}
For any nontrivial linear combination $\sum a_{i}z'(e_{i})$ of elements $z'(e_{i})\in A'_{n,m}$, 
\[
L\Big(\sum a_{i}z'(e_{i})\Big)=2m.
\]
\end{prop}

\begin{proof}
Let $\sum a_{i}z'(e_{i})$ be some linear combination of the elements of $A'_{n,m}$. Assume $L\big(\sum a_{i}z'(e_{i})\big)\neq 2m$. By Theorem \ref{betterbasisforH1}, $L\big(\sum a_{e_{i}}z(e_{i})\big)<2m$. Let $z(e')$ be one of the greatest terms in terms of the $L$-value of $\sum a_{i}z'(e_{i})$. By Corollary \ref{Lfunction}, for every element $z(e_{i})\in A_{n,m}$ there is some linear combination of elements in $A'_{n,m}$ such that $z(e_{i})$ is the unique greatest term. Note the $z(e_{i})$ and $z(e')$ are linearly independent. Thus, $A'_{n,m}$ spans a space of dimension strictly greater than the span of $A_{n,m}$. This contradicts the fact that $A'_{n,m}$ satisfies property $1$ of Theorem \ref{betterbasisforH1}. 
\end{proof}

We extend the standard ordering on $\Z$ to an ordering to $\Z^{n}$ by setting $(a_{1},\dots, a_{n})\succ (b_{1},\dots, b_{n})$ if at the first entry $j$ where $(a_{1},\dots, a_{n})$ and $(b_{1},\dots, b_{n})$ differ $a_{i}> b_{i}$, i.e., we use the lexicographic ordering on $\Z^{n}$. We will use this ordering when we take concatenation products of our basis elements for $H_{1}\big(\text{cell}(n,2);\Q\big)$, and this is the ordering that we consider for the rest of the paper.

In the next section we will use our new basis for $H_{1}\big(\text{cell}(n,2);\Q\big)$ to get free FI$_{d}$-submodules of $H_{k}\big(\text{cell}(n,2);\Q\big)$.

\section{Free FI$_{d}$-submodules of $H_{k}\big(\text{cell}(n,2);\Q\big)$} \label{newHnbasis} 

In this section we find free FI$_{d}$-modules generated by concatenation products of our basis elements for $H_{1}\big(\text{cell}(n,2);\Q\big)$ obtained in Theorem \ref{betterbasisforH1} and certain high-insertion maps in $H_{k}\big(\text{cell}(\bullet,2);\Q\big)$. We calculate the symmetric group representation structure of these subspaces. In the next section we will show that the direct sum of these subspaces is $H_{k}\big(\text{cell}(\bullet,2);\Q\big)$; thus, we will be able to write $H_{k}\big(\text{cell}(n,2);\Q\big)$ as a sum of induced $S_{n}$-representations.

From now on fix a family of sets $A'_{n,m}$ satisfying the properties of Theorem \ref{betterbasisforH1}, such that, for all $n,m$ the elements of $A'_{m,n}$ are indexed in the same way as the elements of $A_{m,n}$.

Given a set $C$ of the form $A_{n,m}$, $A'_{n,m}$, or $B_{n,m}$ set $D(C):=n$. This is the number of labels needed in the symbol that represents an element of $C$. By the correspondence between $\text{cell}(n,w)$ and $\text{Conf}(n,w)$ given in Theorem \ref{confiscell} one can think of $D(C)$ as the number of disks that appear in a cycle representing an element of $C$.

\begin{defn}
An injected cycle arising from $A_{n,m}$, $A'_{n,m}$, or $B_{n,m}$ and an order preserving map $[n]\to S$ where $|S|=n$ is said to be \emph{properly ordered}.
\end{defn}


\begin{exam}
Note that $z'(1|3\,2)$ is a cycle in $A'_{3,1}$, and $z'(2|1|4\,3)$ is a cycle in $A'_{4,2}$, and consider the element of $H_{2}\big(\text{cell}(7,2);\Q\big)$, $z'(1|5\,3|4|2|7\,6)=z'(1|5\,3)|z'(4|2|7\,6)$. Then $z'(1|5\,3)$ is a properly ordered injected cycle arising from $z'(1|3\,2)$ and the order preserving map from $\{1,2,3\}$ into $\{1, 3, 5\}$, and $z'(4|2|7\,6)$ is a properly ordered injected cycle arising from $z'(2|1|4\,3)$ and the order preserving map from $\{1,2,3,4\}$ to $\{2,4,6,7\}$.
\end{exam}

\begin{prop}
Every element $z(e)$ of the basis for $H_{k}\big(\text{cell}(n,2);\Q\big)$ described in Theorem \ref{basistheorem} can be uniquely decomposed into a product of properly ordered injected cycles, i.e., $z(e)=z(e_{1}|\cdots|e_{k})=z(e_{1})|\cdots|z(e_{k})$ such that for $j<k$, the $1$-cycle $z(e_{j})$ arises from a cycle in $A_{m+2,m}$ or $B_{m+2,m}$ for some $n$, and the $1$-cycle $z(e_{k})$ arises from a cycle in $A_{p,m}$ or $B_{p,m}$ for some $p,m$.
\end{prop}

\begin{proof}
This follows immediately from the fact that there is a unique way to decompose $e$ into $e=e_{1}|\cdots|e_{k}$ such that for $j<k$, $e_{j}$ has exactly one block of size $2$ that appears at the far right of the concatenation product, and the concatenation rules force $z(e_{j})$ to be a properly ordered injected cycle arising from a cycle in $A_{m+2,m}$ or $B_{m+2,m}$ for some $m$, and this forces the labeling of the resulting cell $e_{k}$ to the representative cell of a properly ordered injected cycle $z(e_{k})$ arising from a cycle in $A_{p,m}$ or $B_{p,m}$ for some $p,m$.
\end{proof}

\begin{defn}
Let $z(e)=z(e_{1})|\cdots|z(e_{k})$ be an element of basis for $H_{k}\big(\text{cell}(n,2);\Q\big)$ described in Theorem \ref{basistheorem} such that for $j<k$ the $1$-cycle $z(e_{j})$ is a properly ordered injected cycle arising from a cycle in $A_{m+2, m}$ or $B_{m+2, m}$ for some $m$, and $z(e_{k})$ is a properly ordered injected cycle arising from a cycle in $A_{p,m}$ or $B_{p,m}$ for some $p, m$. We say that the \emph{signature} of $z(e)$, denoted $L(z(e))$, is the $k$-tuple of the leading values of the $z(e_{j})$, i.e., $L(z(e)):=(L(z(e_{1})),\dots, L(z(e_{k})))$. 
\end{defn}

\begin{defn}
A \emph{simple barrier} is an properly ordered injected cycle arising from $z(1\,2)$, i.e., the element of $B_{2}$.
\end{defn}

Now we show that the span of certain sets of cycles in $H_{k}\big(\text{cell}(n,2);\Q\big)$ consisting of concatenation products specific families of properly ordered injected cycles form free FI$_{d}$-modules. In the theorem below we write $z'(e)$ for $z(e)$ when $z(e)$ is a properly ordered injected $1$-cycle arising from an element of $B_{n}$ for some $n$, e.g. $z'(1\,2)=z(1\,2)$.

\begin{thm}\label{buildanFId}
Let $k$ be a nonnegative integer, and for $1\le j\le k$ let $C'_{j}$ be a set of $1$-cycles of the form $A'_{m+2,m}$ or $B_{2}$. Let $n=\sum_{j=1}^{k}D(C'_{j})$. Consider the set $C'$ of all cycles in $H_{k}\big(\text{cell}(n,2);\Q\big)$ of the form $z'(e_{1})|\cdots|z'(e_{k})$, where $z'(e_{j})$ is a properly ordered injected cycle arising from an element of $C'_{j}$, such that all $n$ labels are used. Then, the $\Q$-span of $C'$ is isomorphic to the $S_{n}$-representation of the form
\[
\text{Ind}^{S_{n}}_{S_{D(C'_{1})}\times \cdots \times S_{D(C'_{k})}}V_{C'_{k}}\boxtimes\cdots \boxtimes V_{C'_{k}},
\]
where $V_{C'_{j}}=\bigwedge^{2}Std_{m+2}$ if $C'_{j}=A_{m+2,m}$, and $V_{C'_{j}}=Triv_{2}$ if $C'_{j}=B_{2}$.

Moreover, if exactly $d$ of the sets $C'_{j}$ are form $B_{2}$, and we only allow high-insertion to the immediate right of the $1$-cycle that precedes the cycle coming from such a $C'_{j}$, at the far left if $C'_{1}$ is of the form $B_{2}$, as well as high-insertion to the immediate right of the $1$-cycles coming from $C'_{k}$, then $C'$ generates a free FI$_{d+1}$-submodule of $H_{k}\big(\text{cell}(\bullet, 2);\Q\big)$ isomorphic to the free FI$_{d+1}$-module
\[
M^{\text{FI}_{d+1}}\big(\text{Ind}^{S_{n}}_{S_{D(C_{1})}\times \cdots \times S_{D(C_{k})}}V_{C_{k}}\boxtimes\cdots \boxtimes V_{C_{k}}\big).
\]
\end{thm}

Note that our high-insertion maps are well defined, and are the same as the high-insertion maps from \cite{alpert2020generalized}, as, by Theorem \ref{betterbasisforH1}, a properly ordered injected cycle $z'(e_{j})$ arising from some set $C'_{j}$ is a concatenation product linear of a combination of basis elements from Theorem \ref{basistheorem} for $H_{1}\big(\text{cell}(n,2);\Q\big)$, all which all have $1$ barrier. By Lemma \ref{switchsingles} $[\iota_{j}]z'(e_{1})|\cdots|z'(e_{k})$ is homotopic to $z'(e_{1})|\cdots|z'(e_{j})|z(n+1)|z'(e_{j+1})|\cdots|z'(e_{k})$ where $n$ is $\sum_{j=1}^{k}D(C'_{j})$.

Theorem \ref{AlpertsBIGTheorem} proves that $H_{k}\big(\text{cell}(\bullet, 2);\Q\big)$ is a finitely generated FI$_{k+1}$-module; by Proposition \ref{makesmallerFId} we can view it as an FI$_{d}$-module for $d\le k+1$. Thus, it makes sense to speak of an FI$_{d}$-submodule of $H_{k}\big(\text{cell}(\bullet, 2);\Q\big)$.

The idea of this theorem is that we can decompose certain $k$-cycles in $H_{k}\big(\text{cell}(n,2);\Q\big)$ into concatenation products of properly ordered injected $1$-cycles arising from the sets $B_{2}$ or $A'_{m+2,m}$ of Theorem \ref{betterbasisforH1}. See Figure \ref{B2A'31} for an example of decomposition of a $2$-cycle and the allowed high-insertion maps. By restricting which high-insertion maps we allow we get a free FI$_{d+1}$-submodule of $H_{k}\big(\text{cell}(\bullet,2);\Q\big)$. Later in Theorem \ref{wegeteverything} we show that the direct sum of all such FI$_{d}$-submodules is equal to $H_{k}\big(\text{cell}(\bullet,2);\Q\big)$. 

\begin{figure}[h]
\centering
\includegraphics[width = 10 cm]{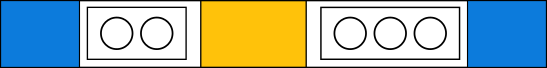}
\caption{The box with $2$ disks in it represents a properly ordered injected cycle arising from the cycle in $B_{2}$, and the box with $3$ disks in it represents a properly ordered injected cycle arising from the cycle in $A'_{3,1}$. The theorem states that if we only allow high-insertion in the blue areas, and not in the yellow, then the result FI$_{2}$-submodule of $H_{2}\big(\text{cell}(\bullet,2);\Q\big)$ is free.}
\label{B2A'31}
\end{figure}

\begin{proof}
By Theorem \ref{betterbasisforH1} the $\Q$-span of $C'_{j}$ forms an $S_{D(C'_{j})}$-representation of the form $V_{C'_{j}}$ when viewed as a subspace of $H_{1}\Big(\text{cell}\big(D(C'_{j}),2\big);\Q\Big)$. Since the labels for $1$-cycles coming from different $C'_{j}$ are distinct by assumption, the $k$ $1$-cycles whose concatenation product is an element of $C'$ come from distinct sets $C'_{j}$, and the $z'(e_{j})$ are properly ordered for all $j$, it follows that as an $S_{n}$-representation the $\Q$-span of $C'$ is a subrepresentation of 
\[
\text{Ind}^{S_{n}}_{S_{D(C'_{1})}\times \cdots \times S_{D(C'_{k})}}V_{C'_{k}}\boxtimes\cdots \boxtimes V_{C'_{k}}.
\]

We count the number of $k$-cycles in $C'$. All the $1$-cycles in $C'_{j}$ use the same number of labels, namely $D(C'_{j})$. Thus, for our choice of injected properly ordered cycle $z'(e_{1})$ there are $\binom{n}{D(C'_{1})}$ choices of the $D(C'_{1})$ labels for $e_{1}$, and once these labels are chosen there are $|C'_{1}|$ properly ordered cycles such that $e_{1}$ is the representative cell of a properly ordered injected cycle $z'(e_{1})$. For $z'(e_{2})$ there are $\binom{n-D(C'_{1})}{D(C'_{2})}$ ways to choose $D(C'_{2})$ labels from the remaining $n-D(C'_{1})$ elements in $[n]$ not being used as labels for $e_{1}$. There are $|C'_{2}|$ ways of ordering these $D(C'_{2})$ elements, such that such that $e_{2}$ is the representative cell of a properly ordered injected cycle $z'(e_{2})$. Continuing in this manner we see that there are 
\[
|C'_{1}|\binom{n}{D(C'_{1})}|C'_{2}|\binom{n-D(C'_{1})}{D(C'_{2})}\cdots |C'_{k}|\binom{n-D(C'_{1})-\cdots-D(C'_{k-1})}{D(C'_{k})}=|C'_{1}|\cdots|C'_{k}|\binom{n}{D(C'_{1}),\dots, D(C'_{k})}
\]
distinct $k$-cycles in $C'$. Note 
\[
\dim\Big(\text{Ind}^{S_{n}}_{S_{D(C'_{1})}\times \cdots \times S_{D(C'_{k})}}V_{C'_{k}}\boxtimes\cdots \boxtimes V_{C'_{k}}\Big)=|C'_{1}|\cdots|C'_{k}|\binom{n}{D(C'_{1}),\dots, D(C'_{k})},
\]
 as $\dim(V_{C'_{j}})=|C'_{j}|$. Thus, if we can show the elements of $C'$ are linearly independent, it follows that $\Q$-span of $C'$ is isomorphic to $\text{Ind}^{S_{n}}_{S_{D(C'_{1})}\times \cdots \times S_{D(C'_{k})}}V_{C'_{k}}\boxtimes\cdots \boxtimes V_{C'_{k}}$ as an $S_{n}$-representation.

We will show that the elements of $C'$ are linearly independent by giving a correspondence with a subset of basis elements for $H_{k}\big(\text{cell}(n,2);\Q\big)$ from Theorem \ref{basistheorem}. If $C'_{j}$ is a set of the form $A'_{m+2,m}$ for some $m$, then let $C_{j}=A_{m+2,m}$. If $C'_{j}$ is of the form $B_{2}$, the let $C_{j}=B_{2}$. Additionally, in this case we set $z(e_{j})=z'(e_{j})$. Given $k$ disjoint subsets $N_{1},\dots, N_{k}$ of $[n]$ such that $|N_{j}|=D(C'_{j})$, consider the $|C'_{1}|\cdots|C'_{k}|$ cycles in $C'$ such that the representative cell $e_{j}$ for $z'(e_{j})$ has labels in $N_{j}$, i.e., the $z'(e_{j})$ are properly ordered injected cycles. By Corollary \ref{Lfunction}, we can write each $z(e_{j})\in C_{j}$ in terms of a linear combination of elements in $C'_{j}$ and cycles $z''$ such that $L\big(z(e_{j})\big)>L(z'')$. Thus, for all $j$ and all $z(e_{j})\in C_{j}$ there is some linear combination of elements of the above form in $C'_{j}$ such that
\[
\sum a_{r_{j}}z'(e_{r_{j}})=z(e_{j})+\sum_{z''|L(z'')<L\big(z(e_{j})\big)}a_{z''}z'',
\]
where the $z'(e_{r_{j}})$ are properly ordered injected cycles in $C'_{j}$ on the labels in $N_{j}$. Thus, given an element $z(e_{1}|\cdots|e_{k})=z(e_{1})|\cdots|z(e_{k})\in H_{k}\big(\text{cell}(n,2);\Q\big)$ of the basis in Theorem \ref{basistheorem} such that $z(e_{j})\in C_{j}$ there is some linear combination of elements in $C'$ such that
\[
\sum a_{r_{1}}\cdots a_{r_{k}}z'(e_{r_{1}})|\cdots|z'(e_{r_{k}})=\sum a_{r_{1}}z'(e_{r_{1}})|\cdots|\sum a_{r_{k}}z'(e_{r_{k}})=z(e_{1})|\cdots|z(e_{k})+\sum a_{z'''}z'''
\]
where the $z'''$ are $k$-cycles such that $L(z''')<L(z(e_{1})|\cdots|z(e_{k}))$.

Note that $z(e_{j})$ arises as a properly ordered injected cycle from $C_{j}$. By the same count that counts the number of elements in $C'$, there are
\[
|C_{1}|\cdots|C_{k}|\binom{n}{D(C_{1}),\dots, D(C_{k})}=|C'_{1}|\cdots|C'_{k}|\binom{n}{D(C'_{1}),\dots, D(C'_{k})},
\]
elements of the form $z(e_{1})|\cdots|z(e_{k})$, where we let $N_{1},\dots, N_{k}$ range over all possible choices for the elements of each $N_{j}$, and $z(e_{j})$ range over properly ordered injected cycles in $C_{j}$ on the labels in $N_{j}$. Moreover, each cycle $z(e_{1})|\cdots|z(e_{k})$ appears as the greatest term, in terms of the ordering coming from $L$, in one such sum of elements of $C'$. Since all such cycles have the same value, in terms of $L$, the smaller terms in the sums cannot cancel out these leading terms. By Theorem \ref{basistheorem} the leading terms are linearly independent. Thus, the $\Q$-span of $C'$ contains a subspace in $H_{k}\big(\text{cell}(n,2);\Q\big)$ of dimension
\[
|C'_{1}|\cdots|C'_{k}|\binom{n}{D(C'_{1}),\dots, D(C'_{k})}.
\]
Therefore, the $\Q$-span of $C'$ is
\[
|C'_{1}|\cdots|C'_{k}|\binom{n}{D(C'_{1}),\dots, D(C'_{k})}
\] 
dimensional, and since it is isomorphic to a subspace of 
\[
\text{Ind}^{S_{n}}_{S_{D(C'_{1})}\times \cdots \times S_{D(C'_{k})}}V_{C'_{k}}\boxtimes\cdots \boxtimes V_{C'_{k}},
\]
of the same dimension, it must be isomorphic to
\[
\text{Ind}^{S_{n}}_{S_{D(C'_{1})}\times \cdots \times S_{D(C'_{k})}}V_{C'_{k}}\boxtimes\cdots \boxtimes V_{C'_{k}},
\]
as an $S_{n}$-representation.

Now we prove that there is a free FI$_{d+1}$-submodule of $H_{k}\big(\text{cell}(\bullet, 2);\Q\big)$, where $d$ is the number the sets $C'_{j}$ of the form $B_{2}$, generated by the $\Q$-span of our set $C'$ and the high-insertion maps to the immediate right of the $1$-cycle preceding the $1$-cycles coming from the $C'_{j}$ of the form $B_{2}$, at the far left if $C'_{1}$ is of the form $B_{2}$, and to the immediate right of the element from $C'_{k}$. 

We claim that there are 
\[
|C'_{1}|\cdots|C'_{k}|\binom{n}{D(C'_{1}),\dots, D(C'_{k})}\binom{p}{n}(d+1)^{p-n}
\]
elements of the form 
\[
z(i_{1})|\cdots|z(i_{n_{1}})|z'(e_{1})|\cdots|z(i_{n_{1}+\cdots+n_{k-1}})|\cdots|z(i_{n_{1}+\cdots+n_{k}})|z'(e_{k})|z(i_{n_{1}+\cdots+n_{k}+1})|\cdots|z(i_{n_{1}+\cdots+n_{k+1}})
\]
where $(n_{1},\dots, n_{k+1})$ is a composition of $p-n$ such that $n_{j}$ is $0$ if $C'_{j}$ is not of the form $B_{2}$, $i_{n_{1}+\cdots+n_{j}+1}>\cdots>i_{n_{1}+\cdots+n_{j}+n_{j+1}}$ for all $j$, and $z'(e_{l})$ is a properly ordered injected cycle arising from a cycle in $C'_{l}$ for all $l$.

To see this, we first show that for a fixed composition $(n_{1},\dots, n_{k+1})$ of $p-n$ such that $n_{j}$ is $0$ if $C'_{j}$ is not of the form $B_{2}$, there are 
\[
|C'_{1}|\cdots|C'_{k}|\frac{p!}{D(C'_{1})!\cdots D(C'_{k})!n_{1}!\cdots n_{k+1}!}
\]
elements of the form
\[
z(i_{1})|\cdots|z(i_{n_{1}})|z'(e_{1})|\cdots|z(i_{n_{1}+\cdots+n_{k-1}})|\cdots|z(i_{n_{1}+\cdots+n_{k}})|z'(e_{k})|z(i_{n_{1}+\cdots+n_{k}+1})|\cdots|z(i_{n_{1}+\cdots+n_{k+1}})
\]
where $i_{n_{1}+\cdots+n_{j}+1}>\cdots>i_{n_{1}+\cdots+n_{j}+n_{j+1}}$ for all $j$, and $z'(e_{l})$ is a properly ordered cycle arising from an element of $C'_{l}$. Then we sum over all such compositions.

Note that there are $\binom{p}{D(C'_{1})+n_{1}}$ ways of choosing $D(C'_{1})+n_{1}$ elements from $[p]$ for the labels that make up $i_{1},\dots, i_{n_{1}}$ and the elements of $e_{1}$. There are $\binom{D(C'_{1})+n_{1}}{D(C'_{1})}|C'_{1}|$ ways of arranging these labels so that the $i_{j}$s are in decreasing order and the symbol $e_{1}$ represents a properly ordered injected cycle arising from an element of $C'_{1}$. There are $\binom{p-D(C'_{1})+n_{1}}{D(C'_{2})+n_{2}}$ ways of choosing $D(C'_{2})+n_{2}$ elements from the remaining $p-D(C'_{1})+n_{1}$ elements of $[p]$, and $\binom{D(C'_{1})+n_{1}}{D(C'_{1})}|C'_{1}|$ ways of arranging these labels so that the $i_{j}$s are in decreasing order and the symbol $e_{2}$ represents a properly ordered injected cycle arising from an element of $C'_{2}$. Repeating this process for the first $k$ $1$-cycles, i.e., the $z'(e_{j})$s, we see that there are
\[
\binom{p}{D(C'_{1})+n_{1}}\binom{D(C'_{1})+n_{1}}{D(C'_{1})}|C'_{1}|\cdots\binom{p-D(C'_{1})-\cdots-D(C'_{d-1})-n_{1}-\cdots-n_{k-1}}{D(C'_{k})+n_{k}}\binom{D(C'_{k})+n_{k}}{D(C'_{k})}|C'_{k}|
\]
ways of choosing our symbols for the 
\[
z(i_{1})|\cdots|z(i_{n_{1}})|z'(e_{1})|\cdots|z(i_{n_{1}+\cdots+n_{k-1}})|\cdots|z(i_{n_{1}+\cdots+n_{k}})|z'(e_{k})
\]
part of our $k$-cycle. Finally, there is $1$ way of putting the remaining $n_{k+1}$ elements in decreasing order. Thus, there are 
\[
\binom{p}{D(C'_{1})+n_{1}}\cdots\binom{p-D(C'_{1})-\cdots-D(C'_{d-1})-n_{1}-\cdots-n_{k-1}}{D(C'_{k})+n_{k}}\binom{D(C'_{1})+n_{1}}{D(C'_{1})}\cdots\binom{D(C'_{k})+n_{k}}{D(C'_{k})}|C'_{1}|\cdots|C'_{k}|
\]
cycles of the form
\[
z(i_{1})|\cdots|z(i_{n_{1}})|z'(e_{1})|\cdots|z(i_{n_{1}+\cdots+n_{k-1}})|\cdots|z(i_{n_{1}+\cdots+n_{k}})|z'(e_{k})|z(i_{n_{1}+\cdots+n_{k}+1})|\cdots|z(i_{n_{1}+\cdots+n_{k+1}})
\]
where $i_{n_{1}+\cdots+n_{j}+1}>\cdots>i_{n_{1}+\cdots+n_{j}+n_{j+1}}$ for all $j$, and $z'(e_{l})$ is a properly ordered cycle arising from an element of $C'_{l}$.

One can check that this is equal to 
\[
|C'_{1}|\cdots|C'_{k}|\frac{p!}{D(C'_{1})!\cdots D(C'_{k})!n_{1}!\cdots n_{k+1}!}.
\]

Note that if $C'_{j}$ not is of the form $B_{2}$ then $n_{j}=0$. Thus, all but $d+1$ of the $n_{j}$ must be $0$, as only the $d$ $n_{j}$ where $C'_{j}$ is of the form $B_{2}$ and $n_{k+1}$ can be nonzero. Let $n_{i_{1}},\dots, n_{i_{k+1}}$ be the $n_{j}$ allowed to be nonzero. Then, for the other $n_{j}$, we have $n_{j}=0$, so $n_{j}!=1$, and 
\[
|C'_{1}|\cdots|C'_{k}|\frac{p!}{D(C'_{1})!\cdots D(C'_{k})!n_{1}!\cdots n_{k+1}!}
=|C'_{1}|\cdots|C'_{k}|\frac{p!}{D(C'_{1})!\cdots D(C'_{k})!n_{i_{1}}!\cdots n_{i_{d+1}}!}.
\]

Additionally, note that there is an obvious bijection between $k+1$-compositions of $p-n$, $(n_{1},\dots, n_{k+1})$, where $n_{j}=0$ if $C'_{j}$ is not of the form $B_{2}$ and $d+1$-compositions of $p-n$, $(n_{i_{1}}, \dots, n_{i_{d+1}})$, by only taking the terms that allowed to be nonzero.

Summing over all $d+1$ compositions $(n_{i_{1}},\dots, n_{i_{d+1}})$ of $p-n$, we see that
\begin{align*}
\sum_{n_{i_{1}}+\cdots+n_{i_{d+1}}=p-n}|C'_{1}|\cdots|C'_{k}|\frac{p!}{D(C'_{1})!\cdots D(C'_{k})!n_{i_{1}}!\cdots n_{i_{d+1}}!}\\
=|C'_{1}|\cdots|C'_{k}|\frac{p!}{D(C'_{1})!\cdots D(C'_{k})!}\sum_{n_{i_{1}}+\cdots+n_{i_{d+1}}}\frac{1}{n_{i_{1}}!\cdots n_{i_{d+1}}!}\\
=|C'_{1}|\cdots|C'_{k}|\frac{p!}{D(C'_{1})!\cdots D(C'_{k})!}\frac{1}{(p-n)!}(d+1)^{p-n}\\
=|C'_{1}|\cdots|C'_{k}|\frac{n!}{D(C'_{1})!\cdots D(C'_{k})!}\binom{p}{n}(d+1)^{p-n}.
\end{align*}
Much of the simplification of this sum is an application of the multinomial theorem.

Thus, there are 
\[
|C'_{1}|\cdots|C'_{k}|\binom{n}{D(C'_{1}),\dots, D(C'_{k})}\binom{p}{n}(d+1)^{p-n}.
\]
$k$-cycles of the form 
\[
z(i_{1})|\cdots|z(i_{n_{1}})|z'(e_{1})|\cdots|z(i_{n_{1}+\cdots+n_{k-1}})|\cdots|z(i_{n_{1}+\cdots+n_{k}})|z'(e_{k})|z(i_{n_{1}+\cdots+n_{k}+1})|\cdots|z(i_{n_{1}+\cdots+n_{k+1}})
\]
where $(n_{1},\dots, n_{k+1})$ is a composition of $p-n$ such that $n_{j}$ is $0$ if $C'_{j}$ is not of the form $B_{2}$, $i_{n_{1}+\cdots+n_{j}+1}>\cdots>i_{n_{1}+\cdots+n_{j}+n_{j+1}}$ for all $j$, and $z'(e_{l})$ is a properly ordered injected cycle arising from a cycle in $C'_{l}$ for all $l$.

Note that the number of $k$-cycles of this form is equal to 
\begin{align*}
\dim\bigg(M^{\text{FI}_{d+1}}\Big(\text{Ind}^{S_{n}}_{S_{D(C'_{1})}\times \cdots \times S_{D(C'_{k})}}V_{C'_{k}}\boxtimes\cdots \boxtimes V_{C'_{k}}\Big)_{p}\bigg)\\=
\dim\Big(\text{Ind}^{S_{n}}_{S_{D(C'_{1})}\times \cdots \times S_{D(C'_{k})}}V_{C'_{k}}\boxtimes\cdots \boxtimes V_{C'_{k}}\Big)\binom{p}{n}(d+1)^{p-n}=
|C'_{1}|\cdots|C'_{k}|\binom{n}{D(C'_{1}), \dots, D(C'_{k})}\binom{p}{n}(d+1)^{p-n}.
\end{align*}
Thus, if we can show that these elements are linearly independent in $H_{k}\big(\text{cell}(p,2);\Q\big)$, and are in our FI$_{d+1}$-submodule, we will have shown that our FI$_{d+1}$-submodule is isomorphic to an FI$_{d+1}$-submodule of the same dimension of a free FI$_{d+1}$-module that contains it, and is hence free.

Next, we show that all the elements of the form
\[
z(i_{1})|\cdots|z(i_{n_{1}})|z'(e_{1})|\cdots|z(i_{n_{1}+\cdots+n_{k-1}})|\cdots|z(i_{n_{1}+\cdots+n_{k}})|z'(e_{k})|z(i_{n_{1}+\cdots+n_{k}+1})|\cdots|z(i_{n_{1}+\cdots+n_{k+1}})
\]
where $(n_{1},\dots, n_{k+1})$ is a composition of $p-n$ such that $n_{j}$ is $0$ if $C'_{j}$ is not of the form $B_{2}$, $i_{n_{1}+\cdots+n_{j}+1}>\cdots>i_{n_{1}+\cdots+n_{j}+n_{j+1}}$ for all $j$, and $z'(e_{l})$ is a properly ordered injected cycle arising from a cycle in $C'_{l}$ for all $l$ are in our FI$_{d+1}$-submodule of $H_{k}\big(\text{cell}(\bullet, 2);\Q\big)$, where we consider the $\Q$-span of $C'$ and only allow high-insertion to the immediate right of the $1$-cycle to the immediate left of the cycle $z'(e_{j})$ if $C'_{j}$ is of the form $B_{2}$, at the far left if $C'_{1}$ is of the form $B_{2}$, and to the immediate right of $z'(e_{k})$. Without loss of generality, we assume that the first $d$ of the $C'_{j}$ are of the form $B_{2}$ and the rest of the $C'_{j}$ are not of this form, as the proof of the general case is similar, but the necessary notation is confusing and unenlightening.

All the elements of the form
\begin{multline*}
z(n+n_{1})|\cdots|z(n+1)|z'(e_{1})|\cdots|z(n+n_{1}+\cdots+n_{d})|\\\dots|z(n+n_{1}+\cdots+n_{d-1}+1)|z'(e_{d})|z'(e'_{d+1})|\cdots|z'(e_{k})|z(p)|\cdots|z(n+n_{1}+\cdots+n_{d}+1),
\end{multline*}
where the labels of the symbol $e_{j}\in [n]$ for all $j$, the sets of labels for $e_{j}$ and $e_{l}$ are disjoint if $j\neq l$, and $z'(e_{j})$ is a properly ordered injected cycle arising from a cycle in $C'_{j}$ for all $j$, are in the FI$_{d+1}$-submodule by definition of the high-insertion maps and $C'$. Given an arbitrary element of the from
\begin{multline*}
z(i_{1})|\cdots|z(i_{n_{1}})|z'(\bar{e}_{1})|\cdots|z(i_{n_{1}+\cdots+n_{d-1}+1})|\cdots|z(i_{n_{1}+\cdots+n_{d-1}+n_{d}})|z'(\bar{e}_{d})|\\
z'(\bar{e}_{d+1})|\cdots|z'(\bar{e}_{k})|z(i_{n_{1}+\cdots+n_{d}+1})|\cdots|z(i_{n_{1}+\cdots+n_{d+1}}),
\end{multline*}
where $i_{n_{1}+\cdots+n_{j}+1}>\cdots>i_{n_{1}+\cdots+n_{j}+n_{j+1}}$ for all $j$, and $z'(\bar{e}_{l})$ is a properly ordered injected cycle arising from a cycle in $C'_{l}$ for all $l$, consider the element of $S_{p}$ that takes $n+n_{1}\mapsto i_{1}, \dots, n+1\mapsto i_{n_{1}}, \dots, n+n_{1}+\cdots+n_{d}+1 \mapsto i_{n+n_{1}+\cdots+n_{d}+1},\dots, n+n_{1}+\cdots+n_{d+1}\mapsto i_{n_{1}+\cdots+n_{d+1}}$, and takes the labels of the symbol of $e_{j}$ to the labels of the symbol $\bar{e}_{j}$, in an order preserving manner for all $j$, e.g., $1|3\,2$ gets sent to $3|6\,4$. This element of $S_{p}$ maps the $k$-cycle we know is in our FI$_{d+1}$-submodule to 
\begin{multline*}
z(i_{1})|\cdots|z(i_{n_{1}})|z'(\bar{e}_{1})|\cdots|z(i_{n_{1}+\cdots+n_{d-1}+1})|\dots|z(i_{n_{1}+\cdots+n_{d}})|z'(\bar{e}_{d})|\\
z'(\bar{e}_{d+1})|\cdots|z'(\bar{e}_{k})|z(i_{n_{1}+\cdots+n_{d+1}})|\cdots|z(i_{n_{1}+\cdots+n_{d}+1}).
\end{multline*}
Since this element arises from the $S_{p}$-action on an element in our FI$_{d+1}$-submodule, it must also be in our FI$_{d+1}$-submodule. Therefore, all elements of the form 
\begin{multline*}
z(i_{1})|\cdots|z(i_{n_{1}})|z'(\bar{e}_{1})|\cdots|z(i_{n_{1}+\cdots+n_{d-1}+1})|\cdots|z(i_{n_{1}+\cdots+n_{d-1}+n_{d}})|z'(\bar{e}_{d})|\\
z'(\bar{e}_{d+1})|\cdots|z'(\bar{e}_{k})|z(i_{n_{1}+\cdots+n_{d}+1})|\cdots|z(i_{n_{1}+\cdots+n_{d+1}}),
\end{multline*}
where $i_{n_{1}+\cdots+n_{j}+1}>\cdots>i_{n_{1}+\cdots+n_{j}+n_{j+1}}$ for all $j$, and $z'(\bar{e}_{l})$ is a properly ordered injected cycle arising from a cycle in $C'_{l}$ for all $l$ are in our FI$_{d+1}$-submodule. If we can show that these elements are linearly independent, then it will follow that our FI$_{d+1}$-submodule is free, as this would prove the dimension of our FI$_{d+1}$-submodule is isomorphic to an FI$_{d+1}$-submodule of a free FI$_{d+1}$-module of the same dimension.

We prove that these elements of our FI$_{d+1}$-submodule of $H_{k}\big(\text{cell}(\bullet, 2);\Q\big)$ are linearly independent by relating them to the basis elements for $H_{k}\big(\text{cell}(p,2);\Q\big)$ from Theorem \ref{basistheorem}. Just like in the proof of Theorem \ref{betterbasisforH1}, let $C_{j}$ be a set of the form $A_{m+2,m}$ if $C'_{j}$ is of the form $A'_{m+2, m}$, and if $C'_{j}$ of the form $B_{2}$, let $C_{j}$ be of the form $B_{2}$. Given a composition of $p-n$ of the form $(n_{1},\dots, n_{d+1})$ let
\begin{multline*}
z(i_{1})|\cdots|z(i_{n_{1}})|z(e_{1})|\cdots|z(i_{n_{1}+\cdots+n_{d-1}+1})|\cdots|z(i_{n_{1}+\cdots+n_{d-1}+n_{d}})|z(e_{d})|\\
z(e_{d+1})|\cdots|z(e_{k})|z(i_{n_{1}+\cdots+n_{d}+1})|\cdots|z(i_{n_{1}+\cdots+n_{d+1}}),
\end{multline*}
be a $k$-cycle in the basis for $H_{k}\big(\text{cell}(p,2);\Q\big)$ from Theorem \ref{basistheorem} where $i_{n_{1}+\cdots+n_{j}+1}>\cdots>i_{n_{1}+\cdots+n_{j}+n_{j+1}}$ for all $j$, and $z(e_{l})$ is a properly ordered injected cycle arising from a cycle in $C_{l}$. This is the element
\[
z(i_{1}|\cdots|i_{n_{1}}|e_{1}|\cdots|i_{n_{1}+\cdots+n_{d-1}+1}|\cdots|i_{n_{1}+\cdots+n_{d-1}+n_{d}}|e_{d}|e_{d+1}|\cdots|e_{k}|i_{n_{1}+\cdots+n_{d}+1}|\cdots|i_{n_{1}+\cdots+n_{d+1}}).
\]
We will show that there is a linear combination of the elements from the previous paragraph whose leading term is this element.

Let $T_{j}$ be the set of labels used in $e_{j}$. By Corollary \ref{Lfunction} there is a linear of properly ordered injected cycles arising from $C'_{j}$ on the elements of $T_{j}$ such that
\[
z(e_{j})+\sum a_{e'}z'(e')=\sum_{\substack{\bar{e}_{j} \text{ such that }z'(\bar{e}_{j})\text{ is a }\\\text{properly ordered injected cycle }\\\text{ arising from }C'_{j}\text{ with labels in }T_{j}}}a_{\bar{e}_{j}}z'(\bar{e}_{j}),
\]
where $L\big(z'(e')\big)<L\big(z(e_{j})\big)$. Thus
\begin{multline*}
z(i_{1})|\cdots|z(i_{n_{1}})|z(e_{1})|\cdots|z(i_{n_{1}+\cdots+n_{d-1}+1})|\cdots|z(i_{n_{1}+\cdots+n_{d-1}+n_{d}})|z(e_{d})|\\
z(e_{d+1})|\cdots|z(e_{k})|z(i_{n_{1}+\cdots+n_{d}+1})|\cdots|z(i_{n_{1}+\cdots+n_{d+1}})+\sum a_{e'''}z'''(e''')
\end{multline*}
is equal to
\begin{multline*}
\sum_{\substack{\bar{e}_{1},\dots, \bar{e}_{k}\text{ such that } z'(\bar{e}_{1}),\dots, z'(\bar{e}_{k})\\\text{ are properly ordered injected cycles }\\\text{arising from } C'_{1},\dots, C'_{k} \text{ with}\\\text{labels in }T_{1},\dots, T_{k}\text{ respectively}}}a_{\bar{e}_{1}}\cdots a_{\bar{e}_{1}}z(i_{1})|\cdots|z(i_{n_{1}})|z'(\bar{e}_{1})|\cdots|z(i_{n_{1}+\cdots+n_{d-1}+1})|\\\cdots|z(i_{n_{1}+\cdots+n_{d-1}+n_{d}})|z'(\bar{e}_{d})|z'(\bar{e}_{d+1})|\cdots|z'(\bar{e}_{k})|z(i_{n_{1}+\cdots+n_{d}+1})|\cdots|z(i_{n_{1}+\cdots+n_{d+1}}),
\end{multline*}
where
\begin{multline*}
L\big(a_{e'''}z'''(e''')\big)<L\big(z(i_{1})|\cdots|z(i_{n_{1}})|z(e_{1})|\cdots|z(i_{n_{1}+\cdots+n_{d-1}+1})|\cdots|z(i_{n_{1}+\cdots+n_{d-1}+n_{d}})|z(e_{d})|\\
z(e_{d+1})|\cdots|z(e_{k})|z(i_{n_{1}+\cdots+n_{d}+1})|\cdots|z(i_{n_{1}+\cdots+n_{d+1}})\big)
\end{multline*}
for all elements $a_{e'''}z'''(e''')$. Thus, for every $k+1$-composition of $p-n$ and every element basis element from Theorem \ref{basistheorem} of the form
\begin{multline*}
z(i_{1})|\cdots|z(i_{n_{1}})|z(e_{1})|\cdots|z(i_{n_{1}+\cdots+n_{d-1}+1})|\cdots|z(i_{n_{1}+\cdots+n_{d-1}+n_{d}})|z(e_{d})|\\
z(e_{d+1})|\cdots|z(e_{k})|z(i_{n_{1}+\cdots+n_{d}+1})|\cdots|z(i_{n_{1}+\cdots+n_{d+1}})
\end{multline*}
there is a linear combination of elements known to be in our FI$_{d+1}$-submodule such that the unique term with the largest signature is this element.

All the elements of the basis for $H_{k}\big(\text{cell}(p,2);\Q\big)$ from Theorem \ref{basistheorem} are linearly independent. Moreover, letting $(n_{1},\dots, n_{d+1})$ vary over all compositions of $p-n$, there are $|C_{1}|\cdots|C_{k}|\frac{n!}{D(C_{1})!\cdots D(C_{k})!}\binom{p}{n}(d+1)^{p-n}$ many elements of the form
\begin{multline*}
z(i_{1})|\cdots|z(i_{n_{1}})|z(e_{1})|\cdots|z(i_{n_{1}+\cdots+n_{d-1}+1})|\cdots|z(i_{n_{1}+\cdots+n_{d-1}+n_{d}})|z(e_{d})|\\
z(e_{d+1})|\cdots|z(e_{k})|z(i_{n_{1}+\cdots+n_{d}+1})|\cdots|z(i_{n_{1}+\cdots+n_{d+1}}),
\end{multline*}
by the same counts as earlier in our proof. Moreover, every the set of linear combinations of the form
\[
\sum_{\substack{\bar{e}_{1},\dots, \bar{e}_{k}\text{ such that } z'(\bar{e}_{1}),\dots, z'(\bar{e}_{k})\\\text{ are properly ordered injected cycles }\\\text{arising from } C'_{1},\dots, C'_{k} \text{ with}\\\text{labels in }T_{1},\dots, T_{k}\text{ respectively}}}a_{\bar{e}_{1}}\cdots a_{\bar{e}_{1}}z(i_{1})|\cdots|z(i_{n_{1}})|z'(\bar{e}_{1})|\cdots|z(i_{n_{1}+\cdots+n_{d-1}+1})|\cdots|z(i_{n_{1}+\cdots+n_{d-1}+n_{d}})|z'(\bar{e}_{d})|
\]
is linearly independent since the leading terms are linearly independent. Since $D(C'_{j})=D(C_{j})$ and $|C'_{j}|=|C_{j}|$ it follows that our basis has a subspace of dimension at least $|C'_{1}|\cdots|C'_{k}|\frac{n!}{D(C'_{1})!\cdots D(C'_{k})!}\binom{p}{n}(d+1)^{p-n}$ in $H_{k}\big(\text{cell}(p, 2);\Q\big)$. As this FI$_{d+1}$-submodule is isomorphic to a submodule of a free FI$_{d+1}$-module of the same dimension for all $p$, our FI$_{d+1}$-submodule of $H_{k}\big(\text{cell}(\bullet, 2);\Q\big)$ must be isomorphic to the entire free FI$_{d+1}$-module and is therefore free.
\end{proof}

Note that $d$ is the number of the sets $C'_{j}$ of the form $B_{2}$, i.e., the number of simple barriers, that constitute an element of $C'$. This follows from the statement of Theorem \ref{buildanFId} as we only allow high-insertion to the immediate right of the $1$-cycles in an element of $C'$ that precedes the cycle coming from a $C'_{l}$ of the form $B_{2}$, at the far left if $C'_{1}$ is of the form $B_{2}$, as well as high-insertion to the immediate right of the $1$-cycle coming from $C'_{k}$. Thus, we allow $d+1$ high-insertions. As every element of $k$-barrier set $C'$ has the same number of simple barriers, we write $S(C')$ to denote the number of simple barriers in any element of $C'$. Similarly, given a $k$-barrier set $C'$, let $D(C')=\sum_{i=1}^{k}D(C'_{i})$ denote the minimum number of labels needed to label an element of $C'$ such no two label are repeated in a symbol.

\begin{defn}
We call sets of the form $C'$ from Theorem \ref{buildanFId} \emph{$k$-barrier sets}, where $k$ is the homological degree of an element in $C'$.
\end{defn}

Let $M^{\text{FI}_{S(C')+1}}(\Q C')$ denote the free FI$_{S(C')+1}$-submodule of $H_{k}\big(\text{cell}(\bullet, 2);\Q\big)$ generated by the $\Q$-span of a $k$-barrier set $C'$ and the allowed high-insertion maps as in Theorem \ref{buildanFId}.

By Theorem \ref{betterbasisforH1} the greatest term in terms of the $L$-value of every element in $C'$ has the same signature. Thus, we can make sense of the signature of a $k$-barrier set.

\begin{defn}
The \emph{signature} of a $k$-barrier set $C'$ is the signature of the greatest element of the basis from Theorem \ref{basistheorem} for $H_{k}\big(\text{cell}(n,2);\Q\big)$ that appears as a summand of an element of $C'$.
\end{defn}

We can extend this definition to arbitrary sums.

\begin{defn}
Given a nontrivial linear combination of $k$-cycles $\sum az$ in the basis of Theorem \ref{basistheorem} in $M^{\text{FI}_{S(C')+1}}(\Q C')_{n}$, we say that the \emph{signature} of $\sum az$ is $\max\big(L(z)\big)$
\end{defn}

\begin{prop}\label{neg1}
Given a $k$-barrier set $C'$, the odd terms of the signature of $C'$ are all $-1$.
\end{prop}

\begin{proof}
Note that the elements of a $k$-barrier set $C'$ are all the cycles of the form $z'(e_{1})|\cdots|z'(e_{k})$, such that $z'(e_{j})$ is a properly ordered injected cycle arising from a set $C'_{j}$ which is a set of the form $B_{2}$ or $A'_{m+2,m}$ for some $m$, and the labels of these $1$-cycles are distinct. By definition $L(C')=\big(L(C'_{1}),\dots, L(C'_{k})\big)$. Note that $L(B_{2})=-1$, and $L(A'_{m+2, m})=2m$. Since none of the $C'_{j}$ can be of the form $B_{m,n}$ for $m\neq 0$, it follows that none of the terms of the signature of $C'$ can be odd numbers greater than $1$.
\end{proof}

We have proven that there are free FI$_{d+1}$-submodules for $d\le k$ in $H_{k}\big(\text{cell}(\bullet,2);\Q\big)$. In the next section we will show that these free FI$_{d+1}$-modules have trivial intersection, and will we use this to decompose $H_{k}\big(\text{cell}(n,2);\Q\big)$ into a direct sum of induced $S_{n}$-representations.

\section{Decomposing $H_{k}\big(\text{cell}(n,2)\big)$}\label{putting it all together}

In this section we show that for every basis element for $H_{k}\big(\text{cell}(n,2);\Q\big)$ from Theorem \ref{basistheorem} there is a unique $k$-barrier set $C'$ such that this basis element is the largest element of this basis appearing as a summand of some element of some $M^{\text{FI}_{S(C')+1}}(\Q C')_{n}$. We use this to decompose $H_{k}\big(\text{cell}(n,2);\Q\big)$ into a direct sum of free FI$_{d}$-modules. This allows us to prove Theorem \ref{decompositionasSp} decomposing $H_{k}\big(\text{Conf}(n,2);\Q\big)$ into a direct sum of induced representations. We then state some results about the multiplicity of the alternating and trivial representations of $S_{n}$ in $H_{k}\big(\text{Conf}(n,2);\Q\big)$, and use the latter result to determine the structure of the rational homology groups of the unordered configuration spaces of unit diameter disks on the infinite strip of width $2$.

\begin{prop}\label{uniquekbarrier}
Given any element $z(e)$ of the basis for $H_{k}\big(\text{cell}(n,2);\Q\big)$ from Theorem \ref{basistheorem} there is a unique $k$-barrier set $C'$ such that $z(e)$ is the unique element with greatest signature of some element in $M^{\text{FI}_{S(C')+1}}(\Q C')_{n}$. If the signature of $z(e)$ is $(l_{1},\dots, l_{k})$, then the signature of $C'$ is $(l'_{1},\dots, l'_{k})$, where $l'_{j}=l_{j}$ if $l_{j}$ is even, and $l'_{j}=-1$ if $l_{j}$ is an odd. Additionally, $\dim\big(M^{\text{FI}_{S(C')+1}}(\Q C')_{n}\big)$ many such basis elements for $H_{k}\big(\text{cell}(n,2);\Q\big)$ from Theorem \ref{basistheorem} appear as the leading term of an element in $M^{\text{FI}_{S(C')+1}}(\Q C')_{n}$.
\end{prop}

\begin{proof}
Every basis element for $H_{k}\big(\text{cell}(n,2);\Q\big)$ from Theorem \ref{basistheorem} is of the form $z(e_{1})|\cdots|z(e_{k})$ where the $z(e_{j})$s are properly ordered injected $1$-cycles on distinct labels arising from a $1$-cycle in $C_{j}$, where for $1\le j\le k-1$, $C_{j}$ is of the form $A_{n_{j},n_{j}-2}$ or $B_{n_{j},n_{j}-2}$, and $C_{k}$ is a set of the form $A_{n_{k}, m_{k}}$ or $B_{n_{k}, m_{k}}$, such that $\sum_{i=1}^{k}n_{i}=p$.

If $z(e_{i})$ is a properly ordered injected cycle arising from an element of $A_{n,m}$, then by Corollary \ref{Lfunction}, there is a linear combination of properly ordered injected cycles $z'(e_{j,i})$ arising from elements of $ A'_{n,m}$ with the same set of labels are $e_{i}$ such that 
\[
z(e_{i})+\sum a''_{e''_{i}}z(e''_{i})=\sum a_{e'_{j,i}}z'(e_{j,i}),
\]
where $L\big(z(e_{i})\big)>L\big(z(e''_{i})\big)$ for all $1$-cycles $z(e''_{i})$ with nontrivial coefficient in the left hand sum. Thus $L\big(z(e_{i})\big)>L\big(\sum a''_{e''_{i}}z(e''_{i})\big)$. It follows that
\[
\sum a_{e'_{j,1}}z'(e_{j,1})|\cdots|\sum a_{e'_{j,k}}z'(e_{j,k})=\sum\cdots\sum a_{e'_{j,1}}\cdots a_{e'_{j,k}}z'(e_{j,1})|\cdots|z'(e_{j,k}),
\]
is a concatenation product of injected cycles on distinct labels arising from the sets $B_{n,m}$ and $A'_{n,m}$ of Theorem \ref{betterbasisforH1}, with unique greatest term in the $L$-value $z(e_{1})|\cdots|z(e_{k})$. By Theorem \ref{buildanFId} each element $z'(e_{j,1})|\cdots|z'(e_{j,k})$ is contained in $M^{\text{FI}_{S(C')+1}}(\Q C')_{n}$, where elements in $C'$ are of the form $z'(\bar{e}_{1})|\cdots|z'(\bar{e}_{k})$, and $z'(\bar{e}_{j})$ is a properly ordered injected cycle arising from an element of $C'_{j}$, such that for $i<k$, $C'_{i}$ is a set of the form $A'_{n_{j},n_{j}-2}$ if the $z'(e_{j,i})$s are properly ordered injected cycle s arising from elements of $A_{n_{j},n_{j}-2}$, and if $z'(e_{j,i})$s are properly ordered injected cycles arising from elemenst of $B_{n_{j},n_{j}-2}$, $C'_{i}$ is of the form $B_{2}$; if the $z'(e_{j,k})$s are properly ordered injected cycles arising from elements of $A_{n_{k},m_{k}}$, $C'_{k}$ is of the form $A'_{n_{k},n_{k}-2}$, and if the $z'(e_{j,k})$s are properly ordered injected cycles arising from elements of $B_{n_{k},m_{k}}$, $C'_{k}$ is of the form $B_{2}$. This also shows that the signature of $C'$ has the form of the proposition statement. Finally, note that that there are exactly as many elements of the form $z(e_{1})|\cdots|z(e_{k})$ as $z'(\bar{e}_{1})|\cdots|z'(\bar{e}_{k})$ as the $e_{j}$s and $\bar{e}_{j}$s are indexed in exactly the same way. Thus, there are least $\dim\big(M^{\text{FI}_{S(C')+1}}(\Q C')_{n}\big)$ many elements of the basis for $H_{k}\big(\text{cell}(n,2);\Q\big)$ from Theorem \ref{basistheorem} that appear as leading terms of an element in $M^{\text{FI}_{S(C')+1}}(\Q C')_{n}$.

To see that $C'$ is the unique $k$-barrier set with $z(e)=z(e_{1})|\cdots|z(e_{k})$ as a leading term, consider the $L$-value of $z(e)$. Let $L\big(z(e)\big)=(l_{1},\dots, l_{k})$. The only way a $z(e)$ is the leading term of some linear combination of elements of $M^{\text{FI}_{S(C')+1}}(\Q C')_{n}$, is if $L(C'')\le L\big(z(e)\big)$, since high-insertion cannot decrease the signature. Moreover, the even terms of $L\big(z(e)\big)$ must be equal to the even terms of $L(C'')$, as the leading terms of the free FI$_{S(C'')+1}$-module can only have $L$-value differing from $L(C'')$ in the odd terms due to the rules we imposed on high-insertion. Additionally, by Proposition \ref{neg1}, given any $k$-barrier set $C''$ the odd terms of $L(C'')$ are all $-1$. There is only one $k$-barrier set $C''$ such that all the even terms of $L(C'')$ are equal to the even terms of $L\big(z(e)\big)$, and all the odd terms are $-1$ namely $C'$. Thus $C'$ is unique. This also, shows that there can be no more than $\dim\big(M^{\text{FI}_{S(C')+1}}(\Q C')_{n}\big)$ many elements of the basis for $H_{k}\big(\text{cell}(n,2);\Q\big)$ from Theorem \ref{basistheorem} that appear as leading terms of an element in $M^{\text{FI}_{S(C')+1}}(\Q C')_{n}$, so there must be exactly this many elements.
\end{proof}

Now we decompose $H_{k}\big(\text{cell}(\bullet,2);\Q\big)$ into a direct sum of free FI$_{d}$-modules.

\begin{thm}\label{wegeteverything}
\[
H_{k}\big(\text{cell}(\bullet,2);\Q\big)=\bigoplus M^{\text{FI}_{S(C')+1}}(\Q C'),
\]
where the direct sum ranges over all $k$-barrier sets $C'$.
\end{thm}

\begin{proof}
By Proposition \ref{uniquekbarrier} every element $z(e)$ the basis of Theorem \ref{basistheorem} for $H_{k}\big(\text{cell}(n,2);\Q\big)$ is the leading term of an element in a FI$_{S(C')+1}$-module generated by some $k$-barrier set $C'$. Thus, $z(e)+\sum a_{e'}z(e')\in M^{\text{FI}_{S(C')+1}}(\Q C')_{n}$, where $L\big(z(e')\big)<L\big(z(e)\big)$ for all nontrivial terms in the sum, and the $z(e')$ are in the basis of Theorem \ref{basistheorem} for $H_{k}\big(\text{cell}(n,2);\Q\big)$. Note that for each $z(e')$ there exists a $k$-barrier set $C'_{e'}$ such that $z(e')$ is the leading term of an element in the FI$_{S(C'_{e'})+1}$-module generated by some $k$-barrier set $C'_{e'}$, i.e., $z(e')+\sum a_{e''}z(e'')\in M^{\text{FI}_{S(C'_{e'})+1}}(\Q C'_{e'})_{n}$, such that $L\big(z(e'')\big)<L\big(z(e')\big)$ for all nontrivial terms in the sum. Note that we can repeat this process a finite number of times, since there are only finitely many elements in $\Z_{\ge -1}^{k}$ less than $L\big(z(e)\big)$, and see that $z(e)\in \sum M^{\text{FI}_{S(C')+1}}(\Q C')_{n}$. Thus, 
\[
H_{k}\big(\text{cell}(n,2);\Q\big)=\sum M^{\text{FI}_{S(C')+1}}(\Q C')_{n}.
\]

Now we prove that 
\[
\sum \dim\big(M^{\text{FI}_{S(C')+1}}(\Q C')_{n}\big)=\dim\Big(H_{k}\big(\text{cell}(n,2);\Q\big)\Big),
\]
as this will prove that $H_{k}\big(\text{cell}(n,2);\Q\big)=\bigoplus M^{\text{FI}_{S(C')+1}}(\Q C')_{n}$. 

By Proposition \ref{uniquekbarrier} there are $\dim\big(M^{\text{FI}_{S(C')+1}}(\Q C')_{n}\big)$ many elements of the basis for $H_{k}\big(\text{cell}(n,2);\Q\big)$ from Theorem \ref{basistheorem} that are leading terms of some element in $M^{\text{FI}_{S(C')+1}}(\Q C')_{n}$. By Proposition \ref{uniquekbarrier}, each element in the basis for $H_{k}\big(\text{cell}(n,2);\Q\big)$ from Theorem \ref{basistheorem} is a leading term of an element of exactly one free FI$_{d}$-module generated by a $k$-barrier set. Thus, the dimension of $H_{k}\big(\text{cell}(n,2);\Q\big)$ is equal to $\sum \dim\big(M^{\text{FI}_{S(C')+1}}(\Q C')_{p}\big)$. Moreover, this is true for all $n$. Therefore, since Theorem \ref{basistheorem} proves that the dimension of $H_{k}\big(\text{cell}(n,2);\Q\big)$ is finite, it follows that 
\[
H_{k}\big(\text{cell}(n,2);\Q\big)=\bigoplus M^{\text{FI}_{S(C')+1}}(\Q C')_{n}
\]
for all $n$. Therefore,
\[
H_{k}\big(\text{cell}(\bullet,2);\Q\big)=\bigoplus M^{\text{FI}_{S(C')+1}}(\Q C')
\]

\end{proof}

\begin{prop}\label{whatsincluded}
Let $C'$ be a $k$-barrier set. Then, $M^{\text{FI}_{S(C')+1}}(\Q C')_{n}=0\subseteq H_{k}\big(\text{cell}(n,2);\Q\big)$ if and only if 
\[
D(C')\le n.
\]
\end{prop}

\begin{proof}
Theorem \ref{buildanFId} proves that $\Q C'$ is an FB-module concentrated in degree $D(C')$. Proposition \ref{buidanFIdfromanFB} shows that the dimension of the $n^{\text{th}}$-degree of the free FI$_{d}$-module generated by a FB-module $W$ concentrated in degree $m$ is
\[
\dim(W_{m})\binom{n}{m}d^{n-m}.
\]
This is non-zero if and only if $n\ge m$.
\end{proof}

\begin{cor}\label{decompositionasSp}
As a representation of $S_{n}$, the homology group $H_{k}\big(\text{Conf}(n,2);\Q\big)$ decomposes as follows
\[
H_{k}\big(\text{Conf}(n,2);\Q\big)\cong\bigoplus_{\substack{(n_{1},\dots, n_{k})\text{ such that}\\n_{i}\ge 2,\text{ }\sum_{j=1}^{k}n_{j}\le n}}\Big(\bigoplus_{\substack{a=(a_{1},\dots, a_{d+1})\text{ such that}\\d\text{ is the number of }n_{i}=2,\\|a|= n-\sum n_{i}}}\big(\text{Ind}^{S_{n}}_{S_{n_{1}}\times\cdots\times S_{n_{k}}\times S_{a}}W_{n_{1}}\boxtimes\cdots\boxtimes W_{n_{k}}\boxtimes \Q\big)\Big),
\]
where $W_{n_{j}}=Triv_{2}$ if $n_{j}=2$, $W_{n_{j}}=Std_{n_{j}}$ if $n_{j}>2$, and $\Q$ denotes the trivial representation of $S_{a}:=S_{a_{1}}\times\cdots \times S_{a_{d+1}}$.
\end{cor}

\begin{proof}
By Theorem \ref{confiscell} $H_{k}\big(\text{Conf}(n,2);\Q\big)$ and $H_{k}\big(\text{cell}(n,2);\Q\big)$ are isomorphic as $S_{n}$-representations. So, by Theorem \ref{wegeteverything}
\[
H_{k}\big(\text{Conf}(n,2);\Q\big)\cong\bigoplus M^{\text{FI}_{S(C')+1}}(\Q C')_{n}.
\]
By Theorem \ref{buildanFId}, we have
\[
M^{FI_{S(C')+1}}(\Q C')_{n}=M^{FI_{S(C')+1}}\left(\text{Ind}^{S_{D(C')}}_{S_{n_{1}}\times\cdots\times S_{n_{k}}}W_{n_{1}}\boxtimes\cdots\boxtimes W_{n_{k}}\right),
\]
where $D(C'_{j})=n_{j}$ and $V_{n_{j}}=Triv_{2}$ if $n_{j}=2$ and $V_{n_{j}}=\wedge^{2}Std_{n_{j}}$ if $n_{j}\ge 2$ as $S_{p}$-representations. By Proposition \ref{freeFId} we can rewrite this as
\[
\bigoplus_{\substack{a=(a_{1},\dots, a_{d+1})\text{ such that}\\d\text{ is the number of }n_{i}=2,\\|a|= n-\sum n_{i}}}\big(\text{Ind}^{S_{n}}_{S_{n_{1}}\times\cdots\times S_{n_{k}}\times S_{a}}W_{n_{1}}\boxtimes\cdots\boxtimes W_{n_{k}}\boxtimes \Q\big).
\]
where $\Q$ denotes the trivial representation of $S_{a}:=S_{a_{1}}\times\cdots \times S_{a_{d+1}}$.

By Proposition \ref{whatsincluded} we only need to consider $k$-barrier sets $C'$ such that $D(C')=\sum_{j=1}^{k}D(C'_{j})\le n$. This is the set of $k$-tuples $(n_{1},\dots, n_{k})$ such that $\sum_{j=1}^{k}n_{j}\le n$ and such that $n_{j}\ge 2$ for all $j$. Moreover, Proposition \ref{whatsincluded} shows that all such $k$-barrier sets are nontrivial summands of $H_{k}\big(\text{cell}(n,2);\Q\big)$. Thus, we see that we can decompose $H_{k}(\text{Conf}(n,2);\Q\big)$ as an $S_{n}$-representation as
\[
H_{k}\big(\text{Conf}(n,2);\Q\big)\cong\bigoplus_{\substack{(n_{1},\dots, n_{k})\text{ such that}\\n_{i}\ge 2,\text{ }\sum_{j=1}^{k}n_{j}\le n}}\Big(\bigoplus_{\substack{a=(a_{1},\dots, a_{d+1})\text{ such that}\\d\text{ is the number of }n_{i}=2,\\|a|= n-\sum n_{i}}}\big(\text{Ind}^{S_{n}}_{S_{n_{1}}\times\cdots\times S_{n_{k}}\times S_{a}}W_{n_{1}}\boxtimes\cdots\boxtimes W_{n_{k}}\boxtimes \Q\big)\Big),
\]
where $W_{n_{j}}=Triv_{2}$ if $n_{j}=2$, and $W_{n_{j}}=Std_{n_{j}}$ if $n_{j}>2$.

\end{proof}

We use Corollary \ref{decompositionasSp} to determined the multiplicity of the alternating and trivial symmetric group representations in $H_{k}\big(\text{conf}(n,2);\Q\big)$. These are examples of multiplicity stability for FI$_{d}$-modules as stated in \cite{ramos2017generalized}, which generalize the multiplicity stability results for FI-modules found in \cite{church2015fi} and \cite{church2014fi}.

\begin{cor}
The alternating representation of $S_{n}$ occurs with multiplicity $1$ in $H_{k}\big(\text{cell}(n,2);\Q\big)$ if $n-3k=0$ or $1$, and with multiplicity $0$ otherwise.
\end{cor}

\begin{proof}
This follows from Corollary \ref{decompositionasSp} and the Littlewood-Richardson Rule, see, for example, \cite[A.8]{fulton2013representation}. The only way for the alternating representation of $S_{n}$ to appear in the decomposition of $H_{k}\big(\text{cell}(p,2);\Q\big)$ given in Corollary \ref{decompositionasSp} is if all of $V_{n_{1}}, \dots, V_{n_{k}}$ are alternating representations of the corresponding $S_{n_{j}}$. By Theorem \ref{betterbasisforH1} this is only possible if $n_{j}=3$ for all $j=1,\dots, k$. In this case $S(C')=0$. Thus, we only need to count the number of times the alternating representation occurs in
\[
\text{Ind}^{S_{n}}_{S_{3}\times\cdots\times S_{3}\times S_{a}}Alt_{3}\boxtimes\cdots\boxtimes Alt_{3}\boxtimes \Q,
\]
where $Alt_{n}$ denotes the alternating representation of $S_{n}$ and $a$ is the $1$-composition of $n-3k$, $(n-3k)$. It follows from Pieri's Rule that the alternating representation of $S_{n}$ occurs if and only if $n-3k=0$ or $1$, see, for example, \cite[A.7]{fulton2013representation}. In these cases the Littlewood--Richardson rule also proves the alternating representation of $S_{n}$ occurs only once.
\end{proof}

\begin{cor}\label{trivappear}
The trivial representation of $S_{n}$ occurs with multiplicity $\binom{n-k}{k}$ in $H_{k}\big(\text{cell}(n,2);\Q\big)$ if $n\ge 2k$, and with multiplicity $0$ otherwise.
\end{cor}

\begin{proof}
This also follows from Corollary \ref{decompositionasSp} and the Littlewood-Richardson rule. The only way for the trivial representation of $S_{n}$ to occur in the decomposition of $H_{k}\big(\text{cell}(n,2);\Q\big)$ given in Corollary \ref{decompositionasSp} is if each of the $V_{n_{1}},\dots, V_{n_{k}}$ is a trivial representation of the corresponding $S_{n_{j}}$. By \ref{betterbasisforH1} this is only possible if $n_{j}=2$ for all $j=1,\dots, k$. In this case $S(C')=k$. Thus, we only need to consider the number of times the trivial representation occurs in 
\[
\bigoplus_{a=(a_{1},\dots, a_{k+1})||a|=n-2k}\big(\text{Ind}^{S_{n}}_{S_{2}\times\cdots\times S_{2}\times S_{a}}Triv_{2}\boxtimes\cdots\boxtimes Triv_{2}\boxtimes \Q\big).
\]
By an application of the Littlewood--Richardson rule, this is equal to the number of $k+1$-compositions of $n-2k$. If $p\ge2k$ this is equal to $\binom{n-2k+k+1-1}{k+1-1}=\binom{n-k}{k}$, and if $n<2k$ then there are no such compositions. \end{proof}

We can use Corollary \ref{trivappear} and the transfer map to determine the structure of the homology groups of the unordered configuration space of $n$ unit-diameter disks on the infinite strip of width $w$.

\begin{cor}
Let $\text{UConf}(n,w)=\text{Conf}(n,w)/S_{n}$ denote the unordered configuration space of $n$ unit-diameter disks on the infinite strip of width $w$. Then, $H_{k}\big(\text{UConf}(n,2);\Q\big)\cong\Q^{\binom{n-k}{k}}$.
\end{cor}

\begin{proof}
Note that $\text{Conf}(n,2)$ is a covering space of $\text{UConf}(n,2)$ with deck group $S_{n}$. The transfer map $\tau_{*}:H_{k}\big(\text{UConf}(n,2);\Q\big)\to H_{k}\big(\text{Conf}(n,2);\Q\big)$ is injective, with image $H_{k}\big(\text{Conf}(n,2);\Q\big)^{S_{n}}$, see, for example, \cite[Proposition 3G.1]{hatcher2005algebraic}. If an element in $H_{k}\big(\text{Conf}(n,2);\Q\big)$ is invariant under the $S_{n}$-action it spans a copy of the trivial representation of ${S_{n}}$. Moreover, every copy of the trivial representation is invariant under the $S_{n}$-action. Thus, $H_{k}\big(\text{Conf}(n,2);\Q\big)^{S_{n}}$ is equal to the sum of all the trivial representations of $S_{n}$ in $H_{k}\big(\text{Conf}(n,2);\Q\big)$. By the previous corollary this has dimension $\binom{n-k}{k}$. Thus, $H_{k}\big(\text{UConf}(n,2);\Q\big)\cong\Q^{\binom{n-k}{k}}$.
\end{proof}

This agrees with Alpert and Manin's Corollary 9.4 \cite[Corollary 9.4]{alpert2021configuration1}. 

\begin{remark}
In general $H_{k}\big(\text{UConf}(\bullet, 2);\Q\big)$ is a free module over the twisted commutative algebra $\Q[x_{0}, \dots, x_{k}]$ where multiplication by $x_{i}$ corresponds to inserting a disk after the $i^{\text{th}}$-barrier, where only the only barriers are images of $z(1\,2)$ under the quotient of the symmetric group action. Thus $H_{k}\big(\text{UConf}(\bullet, 2);\Q\big)$ exhibits a notation of homological stability as $H_{k}\big(\text{UConf}(\bullet, 2);\Q\big)$ is a free $\Q[x_{0}, \dots, x_{k}]$-module generated in degree $2k$.
\end{remark}

We have recovered a decomposition of $H_{k}\big(\text{cell}(n,2);\Q\big)$ into induced $S_{n}$-representations. This decomposition came at a cost. In \cite{alpert2020generalized} Alpert proved that $H_{k}\big(\text{cell}(\bullet,2);\Q\big)$ is a finitely generated FI$_{k+1}$-module. We have used the FI$_{k+1}$-structure, but did away with Alpert's finite generating set. Instead, we found a new generating set that has elements in every large degree used this set to construct smaller free FI$_{d}$-modules. By keeping tracking of our generators and using some basic facts about FI$_{d}$-modules we are able to make our decomposition. In short we have traded finite generation for freedom, losing a more unified view of $H_{k}\big(\text{cell}(n,2);\Q\big)$ to better see the finer details.

\section{Declarations}

\textbf{Conflict of Interest} The author states that there is no conflict of interest.

\bibliographystyle{alpha}
\bibliography{DisksonStrip2Bib}

\end{document}